\documentclass[reqno]{amsart}

\usepackage{amsmath,amsthm,amssymb}
\usepackage{hyperref}
\usepackage{tikz}
\usetikzlibrary{calc}

\newtheorem{thm}{Theorem}
\newtheorem{prop}[thm]{Proposition}
\newtheorem{cor}[thm]{Corollary}
\newtheorem{lem}[thm]{Lemma}
\newtheorem{clm}[thm]{Claim}

\theoremstyle{definition}
\newtheorem{dfn}[thm]{Definition}
\newtheorem{rem}[thm]{Remark}
\newtheorem{ex}[thm]{Example}
\newtheorem{prob}[thm]{Problem}
\newtheorem{ques}[thm]{Question}

\theoremstyle{remark}
\newtheorem*{org}{Organization}
\newtheorem*{ack}{Acknowledgements}
\newtheorem*{nota}{Notation}

\numberwithin{thm}{section}
\numberwithin{equation}{section}

\newcommand{\R}{\mathbb{R}}
\newcommand{\Hcal}{\mathcal{H}}

\newcommand{\GH}{\mathrm{GH}}

\DeclareMathOperator{\In}{In}
\DeclareMathOperator{\diam}{diam}

\usepackage[
  style=numeric,
  sorting=nyt,
  giveninits=true,
  doi=false,
  url=false,
  isbn=false,
  maxnames=10,
]{biblatex}

\AtEveryBibitem{\clearlist{language}}

\DeclareFieldFormat[article,incollection,inproceedings,misc]{title}{#1}
\DeclareFieldFormat[article,incollection,inproceedings,misc]{volume}{\mkbibbold{#1}}
\renewbibmacro{in:}{}
\DeclareFieldFormat{pages}{#1}
\AtEveryBibitem{\clearfield{eprintclass}}
\AtEveryBibitem{%
  \ifentrytype{book}{\clearfield{pages}}{}%
}

\DeclareBibliographyDriver{inproceedings}{%
  \usebibmacro{bibindex}%
  \usebibmacro{begentry}%
  \usebibmacro{author/editor+others}%
  \setunit{\labelnamepunct}\newblock
  \usebibmacro{title}%
  \newunit\newblock
  \printfield{booktitle}%
  \iffieldundef{pages}
    {}
    {\setunit{\addcomma\space}\printfield{pages}}%
  \newunit\newblock
  \usebibmacro{publisher+location+date}%
  \newunit\newblock
  \usebibmacro{finentry}%
}

\DeclareBibliographyDriver{incollection}{%
  \usebibmacro{bibindex}%
  \usebibmacro{begentry}%
  \usebibmacro{author/editor+others}%
  \setunit{\labelnamepunct}\newblock
  \usebibmacro{title}%
  \newunit\newblock
  \printfield{booktitle}%
  \iffieldundef{pages}
    {}
    {\setunit{\addcomma\space}\printfield{pages}}%
  \newunit\newblock
  \usebibmacro{publisher+location+date}%
  \newunit\newblock
  \usebibmacro{finentry}%
}

\DeclareBibliographyDriver{misc}{%
  \usebibmacro{bibindex}%
  \usebibmacro{begentry}%
  \usebibmacro{author/editor+others}%
  \setunit{\labelnamepunct}\newblock
  \usebibmacro{title}%
  \newunit\newblock
  \iffieldundef{note}
    {}
    {{\printfield{note}\setunit{\addcomma\space}}%
      \iffieldundef{eprint}
        {}
      {\usebibmacro{eprint}%
      \setunit{\addcomma\space}%
    }}
  \usebibmacro{date}%
  \newunit\newblock
  \usebibmacro{finentry}%
}

\usepackage{xcolor}

\usepackage{todonotes}
\usepackage{pb-diagram}
\usepackage{appendix}
\newcommand{\Ric}{\mathrm{Ric}}

\title{Busemann and MCP}

\author[T. Fujioka]{Tadashi Fujioka}
\address[T. Fujioka]{Fukuoka University, Fukuoka 814-0180, Japan}
\email{tfujioka@fukuoka-u.ac.jp}

\author[K. Tashiro]{Kenshiro Tashiro}
\address[K. Tashiro]{The University of Osaka, Osaka 560-0043, Japan}
\email{tashiro@math.sci.osaka-u.ac.jp}

\date{\today}

\begin{document}

\begin{abstract}
We study the structure of Busemann spaces of non-positive curvature equipped with measures satisfying the measure contraction property (MCP).
We obtain some rigidity theorems and structure theorems under the assumption of geodesic completeness or non-collapse.
These results partially generalize previous ones by Kapovitch--Ketterer and Kapovitch--Kell--Ketterer for CAT spaces with CD conditions to a Finslerian setting.
\end{abstract}

\maketitle

\tableofcontents

\section{Introduction}

In this paper, we consider the following two conditions simultaneously: the Busemann convexity and the measure contraction property (MCP).
They are, respectively, synthetic notions of upper and lower curvature bounds.
A \textit{Busemann space} (or a \textit{convex space}) is a complete geodesic space such that the distance function is convex along every pair of geodesics (\cite{Bu48, Gr81}).
A metric measure space is said to satisfy the \textit{measure contraction property} MCP($K,N$) if the measure of sets contracting to a point along radial geodesics is controlled by that in the $N$-dimensional model space of constant curvature $K$ (\cite{Oh07, St06}; in general $N$ need not be an integer).
See Section \ref{sec:pre} for the precise definitions.

For a Riemannian manifold $M$, the Busemann convexity is equivalent to having non-positive sectional curvature and infinite injectivity radius, while the MCP($K,n$) for the Riemannian measure is equivalent to having Ricci curvature bounded below by $K$, provided that $n$ is the dimension of $M$.
However, compared to other notions of synthetic curvature bounds, the Busemann convexity and MCP are weak enough to make sense even for Finsler manifolds and sub-Riemannian/sub-Finsler manifolds, respectively.
In fact, the Busemann convexity is weaker than the CAT($0$) condition, and the MCP($K,N$) is weaker than the curvature-dimension condition CD($K,N$) (especially RCD($K,N$)).
See Section \ref{sec:rel} for comparisons of these concepts.

The purpose of this paper is to study the structure of Busemann spaces with measures satisfying the MCP.
We obtain some rigidity theorems and structure theorems for such spaces.
Our work is motivated by and relies on the rigidity result of Andreev \cite{An17} for Busemann G-spaces, as well as the structural results of Kapovitch--Ketterer \cite{KK20} and Kapovitch--Kell--Ketterer \cite{KKK22} for CAT spaces with CD conditions.
We will also review these results in Section \ref{sec:rel}.

\subsection{Main results}

We first prove the following rigidity theorem (cf.\ \cite{An17}).
We say that a geodesic space is \textit{geodesically complete} if every geodesic (i.e., a local shortest path) is extendable infinitely.

\begin{thm}\label{thm:main}
Let $X$ be a geodesically complete Busemann space equipped with a measure $m$ satisfying the measure contraction property {\rm MCP($0,N$)}, where $N\ge 1$.
Then $X$ is isometric to a strictly convex Banach space of dimension $n\le N$.
Furthermore, $m$ is a constant multiple of the $n$-dimensional Hausdorff measure.
\end{thm}

\noindent
This theorem can be interpreted as follows:
Since the Busemann convexity is a generalization of non-positive sectional curvature and the MCP($0,N$) is a kind of non-negative Ricci curvature, combining these two should yield a flat space.
However, since both make sense in the Finsler setting, the conclusion is a Banach space rather than a Hilbert space.

\begin{rem}\label{rem:main}
In Theorem \ref{thm:main}, the dimension $n$ of the Banach space is not necessarily equal to the dimension parameter $N$.
In fact, the Euclidean space $\R^n$ with the Hausdorff measure $\Hcal^n$ satisfies MCP($0,N$) for any $N>n$.
\end{rem}

In Theorem \ref{thm:main}, geodesic completeness rules out the existence of boundary.
In the non-geodesically complete case, we obtain the following generalization under the non-collapsing assumption (cf.\ \cite[Theorem 5.1(ii)]{KK20}).
Here we say that a metric space satisfies the \textit{non-collapsed} MCP($0,n$) if the MCP($0,n$) holds for the $n$-dimensional Hausdorff measure $\Hcal^n$.
In this case, we will use small $n$ for the dimension parameter, which is not necessarily an integer a priori.

\begin{thm}\label{thm:bdry}
Let $X$ be a Busemann space satisfying non-collapsed {\rm MCP($0,n$)}, where $n\ge 1$.
Then $n$ is an integer and $X$ is isometric to a closed convex subset of a strictly convex Banach space of dimension $n$.
\end{thm}

\begin{rem}\label{rem:bdry}
The non-collapsing assumption of Theorem \ref{thm:bdry} is necessary.
For example, for any $N>2$, any sufficiently small closed ball in the hyperbolic plane $\mathbb H^2$ with $\Hcal^2$ satisfies MCP($0,N$).
See Example \ref{ex:hyp} for details (cf.\ \cite[Remark 5.6]{St06}).
This is in contrast to Theorem \ref{thm:main}.
See also Remark \ref{rem:grushin} for the CD case.
\end{rem}

In the course of proving Theorem \ref{thm:bdry}, we obtain the following structure theorem in a more general setting.
Compare with \cite[Corollary 1.2]{KK20} and \cite[Theorem 1.1]{KKK22}.
Note that the curvature assumptions here are local and the lower curvature bound $K$ may be negative.

\begin{thm}\label{thm:mfd}
Let $X$ be a locally Busemann space satisfying local non-collapsed {\rm MCP($K,n$)}, where $K\le0$ and $n\ge 1$.
Then $n$ is an integer and $X$ is an $n$-dimensional topological manifold with boundary.
The manifold interior of $X$ is geodesically convex, has full measure, and coincides with the set of $n$-regular points.
\end{thm}

\noindent Here a point $p\in X$ is called \textit{$n$-regular} if the Gromov--Hausdorff tangent cone of $X$ at $p$ is unique and is isometric to a strictly convex Banach space of dimension $n$ (see Section \ref{sec:tang} for the tangent cone).
As we shall see later in Theorem \ref{thm:inner}, several different notions of ``inner points'' coincide for our space $X$.

\begin{rem}\label{rem:mfd}
The non-collapsing assumption of Theorem \ref{thm:mfd} is expected to be unnecessary.
Indeed, the same conclusion as in Theorem \ref{thm:mfd} holds if $X$ contains at least one manifold point, even in the collapsing case (Corollary \ref{cor:mfdcol}).
The non-collapsing assumption is only used to show the existence of a manifold point (Theorem \ref{thm:noncol}).
See also Remark \ref{rem:cat} for the CAT case.
\end{rem}

\noindent The novelty of Theorem \ref{thm:mfd} lies in the existence of boundary.
Indeed, the geodesically complete (and possibly collapsed) case of Theorem \ref{thm:mfd} is an easy combination of the existing results.
See Corollary \ref{cor:kkbran}.

After settling the manifold structure of $X$ by Theorem \ref{thm:mfd}, one can endow natural coordinates on the interior of $X$ by using the exponential map (see Section \ref{sec:tang} for the exponential map).
We show that this map is an \textit{almost isometry}, i.e., a bi-Lipschitz map whose Lipschitz constants are close to $1$ (Theorem \ref{thm:exp}).
As a consequence, we obtain the following more rigid structure for the interior.

\begin{thm}\label{thm:int}
Let $X$ be a locally Busemann space satisfying local non-collapsed {\rm MCP($K,n$)}, where $K\le0$ and $n\ge 1$.
Then every interior point of $X$ has a neighborhood almost isometric to an open subset of a strictly convex Banach space of dimension $n$.
\end{thm}

\noindent This theorem can be viewed as a local almost rigidity version of Theorem \ref{thm:bdry}.

In the CAT+CD case \cite[Corollary 1.2]{KK20}, \cite[Theorem 1.1]{KKK22} (based on \cite{BN93, Be02, LN19}), the interior admits a Riemannian metric of some regularity.
It is natural to ask whether the interior of our space $X$ admits a Finsler metric (see Problem \ref{prob:fin} and \cite{Po90, Po98}).
However, we will not discuss this problem in this paper, and will address it in future work.

\subsection{Related results and concepts}\label{sec:rel}

\subsubsection*{Andreev's work}

A \textit{G-space}, introduced by Busemann \cite{Bu48, Bu55}, is a qualitative generalization of a Finsler manifold from the perspective of the geometry of geodesics.
More specifically, it is a locally compact, complete, geodesically complete metric space without branching geodesics such that the injectivity radius is locally uniformly bounded below (here the injectivity radius is a maximal radius in which any shortest path admits a unique extension).
Any locally geodesically complete, locally Busemann space with local MCP is, in fact, a G-space (see Proposition \ref{prop:kkbran} for the non-branching property).

In \cite{An14, An17}, Andreev studied the structure of G-spaces satisfying the Busemann convexity (for which the injectivity radius is automatically infinite).
In particular, he proved in \cite{An14} that any G-space satisfying the Busemann convexity is homeomorphic to Euclidean space (cf.\ \cite[Theorem 1.6]{FG25}).
Furthermore, in \cite{An17}, he proved that the tangent cone of such a G-space is isometric to a strictly convex Banach space (Theorem \ref{thm:and}, Remark \ref{rem:and}).
We will use this rigidity result to prove Theorems \ref{thm:main} and \ref{thm:bdry}, as well as the regularity of the manifold interior in Theorems \ref{thm:mfd} and \ref{thm:int}.

\subsubsection*{CAT vs.\ Busemann}

The \textit{{\rm CAT($\kappa$)} condition} is a Riemann-type upper sectional curvature bound for metric spaces defined by triangle comparison (see \cite{BH99} for the basic theory).
For Riemannian manifolds, the CAT($0$) condition is equivalent to the non-positive sectional curvature and infinite injectivity radius, and hence to the Busemann convexity.
However, unlike the CAT($0$) condition, the Busemann convexity also makes sense for Finsler manifolds (\cite{KVK04, KK06, IL19}).
In general, the CAT($0$) condition implies the Busemann convexity, and the converse is true if and only if a Busemann space admits a unique notion of angle (see, for instance, \cite[Exercise 9.81]{AKP24}).

The local structure of CAT spaces was studied by Kleiner \cite{Kl99} (see also \cite{Kr11}), and a finer structure theory in the geodesically complete case was developed by Lytchak--Nagano \cite{LN19, LN22}.
Recently, the first named author and Shijie Gu \cite{FG25} generalized the topological part of the Lytchak--Nagano theory to the setting of geodesically complete Busemann spaces.
However, there is currently no counterpart to Kleiner's result in the non-geodesically complete case, which makes the proofs of Theorems \ref{thm:bdry} and \ref{thm:mfd} more challenging.

\subsubsection*{(R)CD vs.\ MCP}
The \textit{curvature-dimension condition} CD($K,N$), introduced in \cite{LotVil09,St06}, is a synthetic notion of a lower Ricci curvature bound and an upper dimension bound for metric measure spaces in terms of optimal transport.
For Riemannian manifolds, the CD($K,N$) is equivalent to $\Ric\geq Kg$ and $\dim\leq N$, but this notion makes sense even in the Finsler setting, see \cite{Oh09}.
The CD($K,N$) condition (for finite $N$) implies the MCP($K,N$), but the converse is not true: some classes of sub-Riemannian/sub-Finsler manifolds satisfy the MCP($K,N$),
although most of those are known to fail the CD($K,N$) condition, see, e.g., \cite{BarRiz18,BMRT24a,NavPan25} and references therein.
Another remarkable difference is that the MCP does not enjoy the local-to-global property (\cite[Remark 5.6]{St06}), whereas the CD condition does (\cite{CM21}).

The study of the \textit{Riemannian curvature-dimension condition} RCD($K,N$) was initiated in \cite{AmbGigSav14, AmbGigMonRaj15,Gig12}, by adding the so-called infinitesimal Hilbertianity to CD($K,N$) condition. 
Then Finsler-type spaces are excluded, and one achieves the splitting theorem for RCD($0,N$) spaces with finite $N$ (see \cite{Gig13}).
Furthermore, there is a rich structure theory for RCD spaces, such as the rectifiability (\cite{MN19}) and the constancy of dimension (\cite{BS20}).
See also \cite[Theorem 1.13]{Ho24} for a brief summary.
We emphasize that, for MCP or even CD spaces, such nice results cannot be expected, e.g., \cite{KR15, Ma23}.
The best result so far for MCP/CD spaces is the flatness of the tangent cone and the rectifiability by Magnabosco--Mondino--Rossi \cite{MMR25}, under some technical assumptions (see Theorem \ref{thm:mmr}).
This will play an important role in the proofs of Theorems \ref{thm:bdry} and \ref{thm:mfd}.

\subsubsection*{Kapovitch--Kell--Ketterer's work}

In \cite{KK19,KK20,KKK22}, Kapovitch--Ketterer and Kapovitch--Kell--Ketterer developed the structure theory of metric measure spaces satisfying the CAT($\kappa$) and CD($K,N$) conditions.
 Note that CAT spaces are infinitesimally Hilbertian (\cite{DiGigPasSou21}), and thus CAT with CD condition is equivalent to CAT with RCD condition (see also \cite{KK20}).

As explained above, CAT($0$) is stronger than Busemann and CD is stronger than MCP.
Therefore, if we replace ``Busemann'' and ``MCP'' in the main theorems by ``CAT($0$)'' and ``CD'', respectively, then the corresponding (and refined) statements can be derived from the Kapovitch--Kell--Ketterer theory.
For example, compared to Theorems \ref{thm:mfd} and \ref{thm:int}, the interior of a CAT+CD space admits a Riemannian metric of low regularity (\cite[Corollary 1.2]{KK20}, \cite[Theorem 1.1]{KKK22}).
However, this Riemannian character of the CAT condition excludes Finsler-type spaces from the scope of their study.
Our results thus partially generalize theirs in that we are dealing with Finsler-type spaces.
It should be emphasized that some of the results of Kapovitch--Kell-Ketterer apply directly to our Busemann and MCP setting (see Section \ref{sec:factbusemcp}), but others do not.

\subsection{Outline of the proofs}\label{sec:out}

The proof of Theorem \ref{thm:main} is an application of the rigidity theorem of Andreev \cite{An17} mentioned above (Theorem \ref{thm:and}).
We show that our space $X$ is \textit{Busemann concave}, i.e., the opposite inequality to the Busemann convexity holds (Proposition \ref{prop:cone}).
This is analogous to the basic fact that non-positive sectional curvature and non-negative Ricci curvature together imply non-negative sectional curvature (cf.\ \cite{KK20}).
The proof is straightforward --- we first show the homogeneity of the reference measure (Lemma \ref{lem:hom}) and then construct a rigid parallel translation for our Busemann space (Lemma \ref{lem:trans}), which implies the desired Busemann concavity.
Indeed, some arguments can be viewed as rigid versions of those used by Kapovitch--Ketterer \cite{KK20} and Kapovitch--Kell--Ketterer \cite{KKK22} in the study of CAT spaces with CD conditions.
We also provide an alternative proof using the idea from sub-Finsler geometry developed in \cite{Le11} (Theorem \ref{thm:le}).

The proof of Theorem \ref{thm:bdry} is much more involved due to the lack of geodesic completeness.
We first prove Theorem \ref{thm:mfd}.
There are several difficulties in extending the original argument of Kapovitch--Kell--Ketterer in \cite[Section 3]{KKK22}, since it relies heavily on the properties of CAT spaces.
For example, Busemann spaces do not admit an isometric splitting theorem as in the CAT case (cf.\ \cite{Fuk24}).
Furthermore, unlike CAT spaces, the class of Busemann spaces is not closed under limiting operations.
The limit space admits a weaker notion of non-positive curvature called a \textit{convex geodesic bicombing} (see \cite{De16, DL15, DL16}), but in general is not Busemann.
In particular, it is unclear if the tangent cone of a Busemann space is again Busemann.
For these reasons, we shall take a different approach (suggested in \cite[Remark 7.3]{KKK22}) that makes use of the continuity of the tangent cone shown in \cite[Theorem 7.1]{KKK22} (Proposition \ref{prop:kkkconti}).
Together with a basic lemma on Banach spaces (Lemma \ref{lem:ban}), this enables us to prove that the set of manifold points in a Busemann space with MCP is strongly convex (Theorem \ref{thm:conv}).

Another difficulty is that, as explained earlier, the dimension theory of Kleiner \cite{Kl99} developed for CAT spaces is currently not available for Busemann spaces.
In \cite[Theorem 3.15]{KKK22}, this theory allowed Kapovitch--Kell--Ketterer to find a manifold point in a CAT space with CD condition.
Of course, we cannot rely on the structure theory of RCD spaces either.
Instead, we shall use the recent structure result for non-collapsed MCP spaces by Magnabosco--Mondino--Rossi \cite[Theorem 1.2]{MMR25} (Theorem \ref{thm:mmr}).
We show that the non-collapsed MCP implies the almost extendability of geodesics (Claim \ref{clm:noncol1}), which together with the Busemann convexity implies the uniqueness of the tangent cone almost everywhere (Claim \ref{clm:noncol2}).
Then \cite[Theorem 1.2]{MMR25} enables us to find a Banach tangent cone and finally obtain a manifold point (Theorem \ref{thm:noncol}).
The existence of a manifold point, together with the convexity of manifold points proved above, implies the manifold structure as in Theorem \ref{thm:mfd}.

Once Theorem \ref{thm:mfd} is proved, the rest of the proof of Theorem \ref{thm:bdry} reduces to localizing the proof of Theorem \ref{thm:main} with the help of the non-collapsing assumption (Propositions \ref{prop:homloc}, \ref{prop:coneloc}) and proving a Toponogov-type globalization theorem for the Busemann concavity (Proposition \ref{prop:topo}).

Finally, we prove Theorem \ref{thm:int} by generalizing the proof of the local part of Theorem \ref{thm:bdry} to the ``almost MCP($0,n$)'' setting (Theorem \ref{thm:exp}).

\begin{org}
In Section \ref{sec:pre}, we define basic notions including Busemann spaces and the measure contraction property.
In Section \ref{sec:fact}, we recall several facts on Busemann spaces possibly with MCP.
In Section \ref{sec:main}, we prove Theorem \ref{thm:main}.
In Section \ref{sec:mfd}, we prove Theorem \ref{thm:mfd}.
We first establish several results that hold without the non-collapsing assumption and then clarify what improvements the non-collapsing assumption brings.
In Section \ref{sec:bdry}, using Theorem \ref{thm:mfd} and modifying the proof of Theorem \ref{thm:main}, we prove Theorem \ref{thm:bdry}.
In Section \ref{sec:int}, generalizing a part of the proof of Theorem \ref{thm:bdry}, we prove Theorem \ref{thm:int}.
Section \ref{sec:prob} summarizes the remaining open problems.
\end{org}

\begin{ack}
The authors would like to thank Enrico Le Donne for answering their question and introducing them to the reference \cite{HaL23}, and Shouhei Honda for bringing \cite{MMR25} to their attention.
They are also grateful to the anonymous referees for carefully reading the manuscript and providing many detailed comments that helped improve it.
One of their comments led to the discovery of an error in the appendix of an earlier version of this manuscript (which has been removed from the current version).
The first and second named authors were supported by JSPS KAKENHI Grant Numbers 25K23336 and 26K16992, respectively.
The first named author also appreciates the OIST Analysis on Metric Space Unit for supporting his stay in Okinawa.
\end{ack}

\section{Preliminaries}\label{sec:pre}

Here, after introducing some terminology on geodesic spaces, we define Busemann spaces and the measure contraction property.
In particular, we use the geodesic contraction to define Busemann and MCP in a parallel way.
We also discuss the tangent cone of a Busemann space and the exponential map.

\begin{nota}
The distance between two points $x,y$ is denoted by $d(x,y)$.
The shortest path (see the next subsection) between $x$ and $y$ is denoted by $xy$.
The open $r$-ball and closed $r$-ball around $p$, and its boundary are denoted by $B(p,r)$, $\bar B(p,r)$ and $\partial B(p,r)$, respectively.
The $n$-dimensional Hausdorff measure is denoted by $\Hcal^n$.
\end{nota}

\subsection{Geodesic spaces}

A \textit{shortest path} is an isometric embedding of an interval into a metric space.
A \textit{geodesic space} is a metric space such that every two points can be joined by a shortest path.
A \textit{geodesic} is a curve whose restriction to each small interval is a shortest path.
A subset $A$ of a geodesic space is \textit{(geodesically) convex} if for any $x,y\in A$ any shortest path connecting $x$ and $y$ is contained in $A$.

Let $X$ be a geodesic space.
We say that $X$ is
\begin{itemize}
\item \textit{uniquely geodesic} if a shortest path between two given points is unique;
\item \textit{non-branching} if for any two shortest paths, one must be contained in the other when they coincide on a sub-interval;
\item \textit{geodesically complete} if any geodesic extends to a geodesic defined on $\R$.
\end{itemize}
For a complete geodesic space, geodesic completeness reduces to \textit{local geodesic completeness}, defined by the condition that any geodesic is extendable to a geodesic defined on a slightly larger interval (\cite[Proposition 9.1.28]{BBI01}).

Unless otherwise stated, we always assume the completeness of metric spaces.
Furthermore, since every Busemann space is uniquely geodesic, all spaces dealt with in this paper are in fact uniquely geodesic (at least locally, see Section \ref{sec:buse} for more details).
For this reason, we will only consider uniquely geodesic spaces below.

\subsection{Geodesic contraction}\label{sec:cont}

Here we introduce the geodesic contraction to define Busemann and MCP in a parallel way in the next subsections.

Let $X$ be a uniquely geodesic space and fix $p\in X$.
For $x\in X$, we define the \textit{$t$-intermediate point} $x_t=x_t^p$ for $x$ with respect to $p$ as a point on the unique shortest path $px$ such that
\[d(p,x_t)=td(p,x).\]
For a subset $A\subset X$, we define the \textit{$t$-intermediate set} $A_t=A_t^p$ for $A$ with respect to $p$ by
\[A_t:=\{x_t\mid x\in A\}.\]
We also define the \textit{$t$-contraction map} $ \Phi_t=\Phi_t^p:X\to X$ centered at $p$ by
\[\Phi_t(x):=x_t.\]
We will omit the superscript $p$ if there is no ambiguity.

\subsection{Busemann spaces}\label{sec:buse}

We now define Busemann spaces.
For simplicity, we introduce the Busemann convexity under the uniquely geodesic assumption. Note that it is possible to define it without uniqueness of geodesics (and uniqueness is induced as a consequence).
See \cite[Chapter 8]{Pa14} for more details.

\begin{dfn}
A complete uniquely geodesic space $X$ is called a \textit{Busemann space} if for any fixed $p\in X$, we have
\begin{equation}\label{eq:buse}
d(x_t,y_t)\le td(x,y)
\end{equation}
for any $x,y\in X$ and $0\le t\le 1$, where $x_t$ and $y_t$ denote the $t$-intermediate points for $x$ and $y$ with respect to $p$.
We refer to \eqref{eq:buse} as the \textit{Busemann convexity}.
\end{dfn}

\noindent In other words, the $t$-contraction map $\Phi_t$ is $t$-Lipschitz.

\begin{rem}\label{rem:buse}
Another equivalent formulation of a Busemann space is that for any pair of linearly reparameterized shortest paths $\gamma,\eta:[0,1]\to X$, the function
\[t\mapsto d(\gamma(t),\eta(t))\]
is convex on $[0,1]$.
This property plays an essential role in the proof of the rigidity theorem of Andreev \cite{An17} (see also \cite{Bo95}).
However, in this paper we only use the inequality \eqref{eq:buse}, as it provides an opposite estimate to the MCP condition defined in Section \ref{sec:mcp}.
\end{rem}

\begin{ex}
The following are examples of Busemann spaces.
\begin{itemize}
\item CAT($0$) spaces \cite[Proposition 2.2]{BH99}.
\item Strictly convex Banach spaces \cite[Proposition 8.1.6]{Pa14}.
\item Simply-connected Finsler manifolds with Berwald metrics of non-positive flag curvature \cite{KVK04, KK06, IL19}.
\item Gluings and Gromov--Hausdorff limits of Busemann spaces under some reasonable assumptions \cite{An09}.
\end{itemize}
 However, in general the limits of Busemann spaces may fail to be Busemann, e.g., $L^p$ spaces ($1<p<\infty$) converge to $L^1$ or $L^\infty$.
\end{ex}

A \textit{locally Busemann space} is a complete geodesic space such that every point has a neighborhood that is a Busemann space with respect to the restricted metric.
By Remark \ref{rem:buse}, every point in a locally Busemann space has a convex neighborhood.
The Busemann convexity satisfies the following Cartan-Hadamard-type globalization: any simply-connected locally Busemann space is Busemann (\cite{AB90}).

In a Busemann space, every geodesic is a shortest path (\cite[Corollary 8.2.3]{Pa14}).
In particular, if a Busemann space is geodesically complete, then one can extend a shortest path infinitely as a shortest path (not as a geodesic).
Similarly, if a locally Busemann space is locally geodesically complete, then every shortest path near a point $p$ is extendable to length $\epsilon>0$, where $\epsilon>0$ depends only on $p$.

\begin{rem}\label{rem:conc}
The inverse of the Busemann convexity \eqref{eq:buse} is called the \textit{Busemann concavity}, i.e.,
\begin{equation}\label{eq:conc}
d(x_t,y_t)\ge td(x,y),
\end{equation}
where $x_t,y_t$ are $t$-intermediate points for $x,y$ with respect to $p$ in a geodesic space (in this case $x_t,y_t$ are not necessarily unique).
However, there is no counterpart to Remark \ref{rem:buse} for the Busemann concavity, so this notion is not actually ``concave'' in any sense.
See \cite{Ke19} and \cite{HY25} for the study of Busemann concave spaces.
The Busemann concavity plus non-trivial Hausdorff measure implies the (non-collapsed) MCP($0,n$) defined below (\cite[Proposition 2.23]{Ke19}).
\end{rem}

\subsection{Measure contraction property}\label{sec:mcp}

We next define the measure contraction property MCP($K,N$) introduced in \cite{Oh07} in the setting of metric measure spaces, see also \cite[Section 5]{St06}.
Here we mainly focus on the $K=0$ case with unique geodesics (i.e., the contraction map is well-defined), which is our main concern.

A \textit{metric measure space} is a triple $(X,d,m)$ such that $(X,d)$ is a complete separable geodesic space and $m$ is a locally finite, locally positive Borel measure on $X$.
The following is a simplified version of MCP($0,N$) in the uniquely geodesic case (see \cite[Lemma 2.3]{Oh07}).

\begin{dfn}
Let $(X,d,m)$ be a metric measure space.
Suppose, for simplicity, $X$ is uniquely geodesic.
For $N\ge1$ (not necessarily an integer), we say that $(X,d,m)$ satisfies the \textit{measure contraction property} MCP($0,N$) if for any fixed $p\in X$ and any measurable set $A\subset X$ with $0<m(A)<\infty$, we have
\begin{equation}\label{eq:mcp}
m(A_t)\ge t^Nm(A),
\end{equation}
where $0\le t\le 1$ and $A_t$ denotes the $t$-intermediate set for $A$ with respect to $p$.
\end{dfn}

\begin{rem}\label{rem:mcp}
For arbitrary $K\in\R$, one can define MCP($K,N$).
For example, if $K<0$ (and $N>1$), instead of the inequality \eqref{eq:mcp}, we assume
\[m(A_t)\ge\int_{A}t\left(\frac{\sinh(td(p,x)\sqrt{-K/(N-1)})}{\sinh(d(p,x)\sqrt{-K/(N-1)})}\right)^{N-1}dm(x).\]
See \cite[Section 2]{Oh07} for more details.
\end{rem}

\begin{ex}
The following are examples of MCP($K,N$) spaces.
Note that these are not necessarily uniquely geodesic as above.
    \begin{itemize}
        \item RCD($K,N$) spaces, CD($K,N$) spaces \cite{St06}; in particular, Riemannian/Finsler manifolds with $N$-weighted Ricci curvature $\ge K$ \cite{Oh09}.
        \item $n$-dimensional Banach spaces with the Lebesgue measure satisfy MCP($0,n$) cf.\ \cite[Example 29.16]{Vil}.
        \item Some sub-Riemannian Carnot groups \cite{Jui09,Rif13, Ri16, BarRiz18}, and some sub-Finsler Heisenberg groups \cite{BorTas23,BMRT24a}.
        \item Measured Gromov--Hausdorff limits of MCP($K,N$) spaces \cite[Theorems 6.8, 6.11]{Oh07}
    \end{itemize}
   However, in general the gluings of MCP spaces may fail the MCP. For instance, the half line equipped with a measure degenerated at the boundary $([0,\infty),|\cdot|,|x|\cdot \Hcal^1)$ is an MCP$(0,2)$ space, while the gluing of two copies of it at $0$ does not satisfy the MCP$(0,N)$ for any $N$. See also \cite{Riz18}.
\end{ex}

A metric space is said to satisfy the \textit{non-collapsed} MCP($K,n$) if it satisfies MCP($K,n$) for the $n$-dimensional Hausdorff measure $\Hcal^n$.
In this case, we will use small $n$ to emphasize it.

A metric measure space is said to satisfy the \textit{local} MCP($K,N$) if every point has a convex neighborhood that satisfies the MCP($K,N$) with respect to the restricted metric and measure.
As mentioned in Section \ref{sec:buse}, the existence of a convex neighborhood and the uniqueness of shortest paths in such neighborhoods are always satisfied in the setting of locally Busemann spaces.
Note that the MCP does not enjoy the local-to-global property (\cite[Remark 5.6]{St06}), whereas the curvature-dimension condition does (\cite{CM21}).

The MCP($K,N$) implies the \textit{($K,N$)-Bishop--Gromov inequality} (\cite[Theorem 5.1]{Oh07}).
For example, if $K=0$, this means that for any $p\in X$, the function 
\[m(B(p,r))/r^N\]
is non-increasing in $r>0$; in general, a corresponding comparison value defined by $r$, $K$, and $N$ appears in the denominator.
In particular, the measure contraction property implies the uniform local doubling property.
Here, a metric space is \textit{uniformly locally doubling} if for all $R>0$, there exists a constant $C$ such that for any $0<r<R$, any $r$-ball is covered by at most $C$ balls of radius $r/2$.
In particular, any MCP space is proper, i.e., every closed ball is compact.

The Hausdorff dimension of an MCP($K,N$) space is always less than or equal to the dimension parameter $N$, see \cite[Corollary 2.7]{Oh07}.

\subsection{Tangent cones}\label{sec:tang}

Finally, we recall two different notions of tangent cones for Busemann spaces and define the exponential map in the non-branching case.
This will be used in the proofs of Theorems \ref{thm:mfd}, \ref{thm:bdry}, and \ref{thm:int}, but is not necessary for Theorem \ref{thm:main}.

We first define the Gromov--Hausdorff tangent cone.
Let $X$ be a locally doubling, locally Busemann space.
For $p\in X$ and $\lambda_i\to\infty$, the pointed Gromov--Hausdorff limit
\[(T_p^{(\lambda_i)}X,o):=\lim_{i\to\infty}(\lambda_i X,p)\]
is called the \textit{Gromov--Hausdorff tangent cone} of $X$ at $p$ for scale $(\lambda_i)$ (if it exists).
 Here $\lambda X$ denotes the rescaled space $X$ with the metric multiplied by $\lambda$.
By the local doubling condition, after passing to a subsequence, the limit always exists, but is not necessarily unique.
If there is no ambiguity, we will omit the superscript $(\lambda_i)$ and denote it by $T_pX$.
We call $o$ the \textit{apex} of $T_pX$.

Next we define the geodesic tangent cone (cf.\ \cite{Ke19, An14, KKK22}).
Let $X$ be a locally Busemann space.
Let $\Gamma_p$ be the set of shortest paths emanating from $p\in X$.
For a product space $\Gamma_p\times[0,\infty)$, we define a pseudo metric $d^*$ by
\[d^*((\gamma,a),(\eta,b)):=\lim_{t\to0}\frac{d(\gamma(at)),\eta(bt))}{t},\]
where $(\gamma,a),(\eta,b)\in\Gamma_p\times[0,\infty)$.
The Busemann convexity \eqref{eq:buse} ensures the existence of the limit.
The completion of the metrization of $\Gamma_p\times[0,\infty)$ is called the \textit{geodesic tangent cone} of $X$ at $p$ and is denoted by $T_p^gX$.
By abusing the notation, we denote by $o$ the equivalent class of $(\gamma,0)\in\Gamma_p\times[0,\infty)$ and call it the \textit{apex} of $T_p^gX$.
Clearly, for any $\lambda>0$, we have
\[d^*((\gamma,\lambda a),(\eta,\lambda b))=\lambda d^*((\gamma,a),(\eta,b)),\]
where $(\gamma,a),(\eta,b)\in\Gamma_p\times[0,\infty)$.
This defines the cone structure of $T_p^gX$.

In general, the geodesic tangent cone $T_p^gX$ isometrically embeds into any Gromov--Hausdorff tangent cone $T_pX$ (\cite[Lemma 3.4(i)]{KKK22}).
This embedding is surjective in the following case (\cite[Lemma 3.4(ii)]{KKK22}; see also \cite[Corollary 5.7]{LN19}).
Although the original statement is for CAT spaces, the proof also works for Busemann spaces.

\begin{lem}\label{lem:tang}
Let $X$ be a locally doubling, locally Busemann space and $p\in X$.
Suppose there exists $\epsilon>0$ such that any shortest path emanating from $p$ extends to length $\epsilon$.
Then the geodesic tangent cone $T_p^gX$ is isometric to any Gromov--Hausdorff tangent cone $T_pX$.
In particular, $T_pX$ is unique.
\end{lem}

\noindent The following example shows that the conclusion of Lemma \ref{lem:tang} may fail if the uniform extendability assumption is dropped.

\begin{ex}
Let $\R$ be the real line.
Take sequences of positive numbers $\delta_i,\epsilon_i$ converging to $0$ such that
\[\delta_i/\epsilon_i\to 0,\quad \epsilon_{i+1}/\delta_i\to 0.\]
Attach intervals $[-\delta_i,\delta_i]$ to $\R$ by identifying $0\in[-\delta_i,\delta_i]$ with $\delta_i\in\R$.
Similarly, attach intervals $[0,\epsilon_i]$ to $\R$ by identifying $0\in[0,\epsilon_i]$ with $\epsilon_i\in\R$.
We denote by $X$ the resulting space, which is a doubling Busemann space.
Set $p:=0\in\R\subset X$.
Then $T_p^{(\delta_i^{-1})}X$ is $\R$ with $[-1,1]$ attached by identifying $0\in[-1,1]$ with $1\in\R$.
Similarly, $T_p^{(\epsilon_i^{-1})}X$ is $\R$ with $[0,1]$ attached by identifying $0\in[0,1]$ with $1\in\R$.
On the other hand, $T_p^gX$ is isometric to $\R$.
\end{ex}

\noindent  Nevertheless, it is possible to relax the extendability assumption in Lemma \ref{lem:tang} to what we call ``almost extendability'' (Claim \ref{clm:noncol2}).
One remarkable observation is that this milder assumption follows from the non-collapsed MCP condition (Claim \ref{clm:noncol1}).
These claims are the core of Theorem \ref{thm:noncol}.

\begin{rem}\label{rem:doub}
Under the uniform extendability of geodesics at $p$ as in Lemma \ref{lem:tang}, the local doubling condition is equivalent to local compactness; see \cite[Proposition 3.1]{FG25} for the proof.
\end{rem}

\begin{rem}\label{rem:tang}
Let $X$ be a locally doubling, locally Busemann space and $p\in X$.
Since the Busemann convexity is not preserved by taking limits in general, it is unknown whether the Gromov--Hausdorff tangent cone $T_pX$ is Busemann.
However, $T_pX$ admits a weaker variant of non-positive curvature called a \textit{convex geodesic bicombing}.
That is, the distance function of $T_pX$ satisfies the convexity condition \eqref{eq:buse} only for the shortest paths that arise as limits of the shortest paths of $(\lambda_iX,p)$.
For more details, see \cite{De16, DL15, DL16} and references therein (see also \cite[Section 10.1]{Kl99} and \cite{AS19}).

On the other hand, the MCP condition is preserved by the (pointed) measured Gromov--Hausdorff convergence (\cite[Theorem 6.11]{Oh07}).
Hence, if $X$ satisfies local MCP($K,N$), then $T_pX$ satisfies MCP($0,N$) for some rescaled limit measure.
\end{rem}

Now we define the exponential map in the non-branching case.
Let $X$ be a locally doubling, locally Busemann space without branching geodesics and $p\in X$.
Suppose any shortest path emanating from $p$ extends to uniform length, as in Lemma \ref{lem:tang} (in particular, $T_pX=T_p^gX$).
Since $X$ is non-branching, one can define $\exp_p:B(o,r)\to B(p,r)$
by
\[\exp_p((\gamma,t)):=\gamma(t)\]
for $(\gamma,t)\in B(o,r)\subset T_p^gX$, provided $r>0$ is sufficiently small.
Then $\exp_p$ is a homeomorphism.
See \cite[Lemmas 3,4]{An14} for the proof.
In Theorem \ref{thm:exp}, under the additional assumption of non-collapsed MCP, we will show that $\exp_p$ is a bi-Lipschitz homeomorphism with Lipschitz constants close to $1$.

\section{Facts}\label{sec:fact}

Here we collect several facts on Busemann spaces possibly with MCP condition.
For the proof of Theorem \ref{thm:main}, we only need Theorem \ref{thm:and}.
The other results are used in the proofs of Theorems \ref{thm:mfd} and \ref{thm:bdry}, so the reader can skip them until needed.

\subsection{Facts on Busemann}\label{sec:factbuse}

First we recall facts on Busemann spaces without MCP condition and prove their corollaries.

The following rigidity theorem of Andreev \cite{An17} is the key ingredient in the proofs of Theorems \ref{thm:main} and \ref{thm:bdry} (and part of Theorems \ref{thm:mfd} and \ref{thm:int}).
We say that a Busemann space $X$ is \textit{of cone-type} if there exists $p\in X$ such that
\begin{equation}\label{eq:cone}
d(x_t,y_t)=td(x,y)
\end{equation}
for any $x,y\in X$ and $0\le t\le 1$, where $x_t,y_t$ denote the $t$-intermediate points for $x,y$ with respect to $p$, respectively.
In other words, not only the Busemann convexity \eqref{eq:buse}, but also the Busemann concavity \eqref{eq:conc} holds for $p$.

\begin{thm}[{\cite[Theorem 1]{An17}}]\label{thm:and}
Let $X$ be a locally compact, geodesically complete Busemann space without branching geodesics.
Suppose $X$ is of cone-type.
Then $X$ is isometric to a strictly convex Banach space of finite dimension.
\end{thm}

\begin{rem}
As mentioned in Section \ref{sec:rel}, the original theorem of Andreev \cite{An17} is stated in terms of G-spaces, but in this paper we will not use this terminology in order to clarify the conditions we are dealing with.
\end{rem}

\begin{rem}\label{rem:and}
Every tangent cone of a locally compact, locally geodesically complete, locally Busemann space without branching geodesics satisfies the assumption of Theorem \ref{thm:and}.
See \cite[Theorem 4]{An14} for the proof.
Note that the two definitions of tangent cones coincide by Lemma \ref{lem:tang}.
\end{rem}

Before moving on to the next results, we define several notions of inner points for Busemann spaces (some of which are taken from \cite[Definition 1.4]{LS07}).
Let $X$ be a locally Busemann space and $p\in X$.
We say that $p$ is
\begin{itemize}
\item a \textit{manifold point}, or an \textit{$n$-manifold point}, if $p$ has a neighborhood that is homeomorphic to $\R^n$;
\item a \textit{regular point}, or an \textit{$n$-regular point}, if the Gromov--Hausdorff tangent cone of $X$ at $p$ is unique and isometric to a strictly convex Banach space of dimension $n$;
\item a \textit{topologically inner point} if for any sufficiently small $r>0$, the punctured ball $B(p,r)\setminus\{p\}$ is non-contractible (note that its homotopy type is independent of $r$, thanks to the geodesic contraction);
\item a \textit{geometrically inner point} if there exists $\epsilon>0$ such that any shortest path ending at $p$ is extendable beyond $p$ to length $\epsilon$;
\end{itemize}

Later in Theorem \ref{thm:inner}, we prove that all the above notions are equivalent under the assumptions of MCP and the existence of a manifold point.
Here we recall some of these equivalences that hold without such additional assumptions.

The first is due to Lytchak--Schroeder \cite[Theorem 1.5(1)]{LS07} (see also \cite[Lemma 3.6]{FG25} for a weaker version).
Although the original statement is formulated for CAT spaces, the proof applies verbatim to Busemann spaces.

\begin{prop}\label{prop:ls}
Let $X$ be a locally Busemann space and $p\in X$.
If $p$ is topologically inner, then $p$ is a geometrically inner point.
\end{prop}

The second is due to Kapovitch--Kell--Ketterer, and is included in the proof of \cite[Proposition 3.11]{KKK22} for a special class of CAT spaces.

\begin{prop}\label{prop:kkkreg}
Let $X$ be a locally Busemann space and $p\in X$.
If $p$ is a regular point, then $p$ is topologically inner.
Moreover, the same conclusion holds under a slightly weaker assumption that $T_pX$ is unique and isometric to a finite-dimensional Banach space (not necessarily strictly convex).
\end{prop}

\noindent
The proof is essentially the same as the original one, but requires several minor modifications.
For the convenience of the reader, we include a complete argument.
Before the proof, we prepare some basic facts on Busemann spaces.

First, let $X$ be a Busemann space.
Then
\begin{enumerate}
\item any (non-continuous) map $f:K\to X$ from a finite simplicial complex $K$ such that $\diam f(\sigma)\le\epsilon$ for every simplex $\sigma\subset K$ can be $2\epsilon$-approximated by a continuous map $g:K\to X$;
\item any two continuous maps $f,g:Y\to X$ from a topological space $Y$ are homotopic.
\end{enumerate}
For the proof of (1), consider the barycentric subdivision $K'$ of $K$ and first define $g\equiv f$ on the set of vertices of $K'$.
Since any simplex $\sigma\subset K$ is a cone over its boundary $\partial\sigma$, where the vertex of the cone is identified with the barycenter of $\sigma$, one can define $g$ inductively on the skeleta of $K'$ by using the unique shortest paths of $X$.
Here, since metric balls in $X$ are convex, $f(x)$ and $g(x)$ remain in the same $\epsilon$-ball for each $x\in K$.
By the triangle inequality, $g$ is $2\epsilon$-close to $f$.
For the proof of (2), just connect $f(y)$ and $g(y)$ for any $y\in Y$ by the unique shortest path of $X$.
Note that Properties (1) and (2) are also true for any normed space $X$, even if it is not strictly convex, by using affine geodesics.

Next, suppose $T_pX$ is a finite-dimensional Banach space for $p\in X$.
Consider the Gromov--Hausdorff convergence $(\lambda X,p)\to(T_pX,o)$ as $\lambda\to\infty$.
Then
\begin{enumerate}
\item[(3)] any limit of shortest paths in $\lambda X$ is an affine geodesic of $T_pX$.
\end{enumerate}
Indeed, if $T_pX$ is strictly convex (i.e., uniquely geodesic), this is trivial.
In general, this follows from the fact that affine geodesics are only curves in a normed space on which the distance function from an arbitrary point is convex, see \cite[Theorem 3.3]{DL15}.
By Remark \ref{rem:buse}, the limit shortest path in (3) satisfies this condition.

In particular, we will use Property (1) instead of the barycentric construction of Kleiner \cite{Kl99} used in the original proof (cf.\ \cite[Lemma 3.1]{Kr11}).
Note also that the original proof only concerns the Euclidean tangent cone, for which Property (3) is trivial.

\begin{proof}[Proof of Proposition \ref{prop:kkkreg}]
The proof is a minor modification of the original one in \cite[Proposition 3.11]{KKK22}.
For $0<r_1<r_2$, we denote by $A(p;r_1,r_2)$ the metric annulus $\bar B(p,r_2)\setminus B(p,r_1)$ centered at $p$.
We also use the same symbol $\epsilon_i$ to denote possibly different sequences of positive numbers converging to $0$.

Let $o$ denote the apex of the tangent cone $T_pX$.
Setting $r_i:=1/2^i$, we consider $\epsilon_i$-Gromov--Hausdorff approximations
\[f_i:r_i^{-1}B(p,10r_i)\to B(o,10),\quad g_i:B(o,10)\to r_i^{-1}B(p,10r_i)\]
such that $f_i\circ g_i$ and $g_i\circ f_i$ are $\epsilon_i$-close to the identities.

Let $S:=\partial B(o,3/4)\subset T_pX$, which is homeomorphic to $S^{n-1}$.
We regard $S$ as a continuous map from $S^{n-1}$ (with a sufficiently fine triangulation).
Using Property (1) before the proof, one can define a continuous map $S_i:S^{n-1}\to A(p;r_{i+1},r_i)$ that approximates $g_i\circ S$.
Then the homology class satisfies
\begin{equation}\label{eq:chain1}
[S_i]\neq 0\in H_{n-1}(A(p;r_{i+1},r_i))
\end{equation}
for any large $i$.
Indeed, if $[S_i]=0$, there exists an $n$-chain $C_i$ in $A(p;r_{i+1},r_i)$ with $\partial C_i=S_i$.
Applying Property (1) to $f_i\circ C_i$, one can construct an $n$-chain $C$ in $A(o;1/2,1)$ such that $\partial C$ is $\epsilon_i$-close to $S$.
Using Property (2), one can show that $\partial C$ is homologous to $S$ in $B(o,1)\setminus\{o\}$.
This implies $[S]=0$ in $H_{n-1}(B(o,1)\setminus\{o\})$, a contradiction.

We prove that
\[[S_i]\neq 0\in H_{n-1}(B(p,r_i)\setminus\{p\})\]
for any large $i$, which implies that $p$ is a topologically inner point.
Suppose this does not hold for some large $i_0$.
Then there exists an $n$-chain $C_{i_0}$ in $B(p,r_{i_0})\setminus\{p\}$ such that $S_{i_0}=\partial C_{i_0}$.
Since the support of $C_{i_0}$ is compact, it is contained in $A(p;\delta,r_{i_0})$ for some $\delta>0$.
Let $\Phi_i:X\to X$ denote the $r_i$-contraction map centered at $p$.
Choosing $j\gg1$ such that $r_j<\delta$ and using the geodesic contraction centered at $p$, we see that
\begin{equation}\label{eq:chain2}
[\Phi_{j-{i_0}}(S_{i_0})]=0\in H_{n-1}(A(p;r_{j+1},r_j)).
\end{equation}
On the other hand, we will show that
\begin{equation}\label{eq:chain3}
[\Phi_{j-i_0}(S_{i_0})]=\pm[S_j]\in H_{n-1}(A(p;r_{j+1},r_j)).
\end{equation}
Combining \eqref{eq:chain1}, \eqref{eq:chain2}, and \eqref{eq:chain3}, we get a contradiction.

We prove \eqref{eq:chain3} by induction on $j$.
The base case $j=i_0$ is trivial.
To prove the induction step, it suffices to show that
\begin{equation}\label{eq:chain4}
[\Phi(S_j)]=\pm[S_{j+1}]\in H_{n-1}(A(p;r_{j+2},r_{j+1})),
\end{equation}
where $\Phi$ is the $1/2$-contraction map centered at $p$ (so $\Phi_i$ is the $i$-th iterate of $\Phi$).

Let us prove \eqref{eq:chain4}.
We show that $f_{j+1}(\Phi(S_j))$ is $\epsilon_j$-close to $S$.
Then $\Phi(S_j)$ is $\epsilon_j$-close to $S_{j+1}$, and thus \eqref{eq:chain4} follows from Property (2).

Define $\Psi:T_pX\to T_pX$ by $\Psi(v)=v/2$ for any $v\in T_pX$.
In other words, $\Psi$ is the ``$1/2$-contraction map'' centered at $o$ defined by affine geodesics.
By Property (3), we see that $f_{j+1}$ almost commutes with $\Phi$ and $\Psi$.
In particular, $f_{j+1}(\Phi(S_j))$ is $\epsilon_j$-close to $\Psi(f_{j+1}(S_j))$.

Observe that the composition
\[\Psi\circ f_{j+1}:r_j^{-1}B(p,5r_j)\to B(o,5)\]
gives another Gromov--Hausdorff approximation, possibly different from $f_j$.
This, together with the uniqueness of the Gromov--Hausdorff tangent cone, implies that $\Psi\circ f_{j+1}$ and $f_j$ are $\epsilon_j$-close modulo an isometry of $T_pX$ fixing $o$.
Therefore,
\[f_{j+1}(\Phi(S_j))\approx \Psi(f_{j+1}(S_j))\approx' f_j(S_j)\approx S\]
where $\approx$ represents $\epsilon_j$-closeness and $\approx'$ represents ``$\epsilon_j$-closeness modulo an isometry fixing $o$''.
Since $S$ is a metric sphere centered at $o$, we obtain $f_{j+1}(\Phi(S_j))\approx S$, as desired.
This completes the proof.
\end{proof}

From the above results, we derive the following corollary.

\begin{cor}\label{cor:reg}
Let $X$ be a locally doubling, locally Busemann space without branching geodesics.
Then $p\in X$ is an $n$-manifold point if and only if it is an $n$-regular point.
\end{cor}

\begin{proof}
Suppose $p$ is an $n$-manifold point.
By Proposition \ref{prop:ls}, $X$ is locally geodesically complete around $p$.
By Theorem \ref{thm:and} and Remark \ref{rem:and}, $p$ is a regular point.
Moreover, using the exponential map defined in Section \ref{sec:tang}, we see that it is an $n$-regular point.

Suppose $p$ is an $n$-regular point.
By Propositions \ref{prop:kkkreg} and \ref{prop:ls}, $p$ is a geometrically inner point.
This, together with the assumption that $X$ is non-branching, implies that any shortest path emanating from $p$ extends to uniform length, as assumed in Lemma \ref{lem:tang}.
(Indeed, for any shortest path $px$, extend it beyond $p$ and let $q$ be an endpoint.
Since $p$ is geometrically inner, we can again extend $qp$ beyond $p$ to uniform length.
By the non-branching assumption, this is actually an extension of $px$.)
By Lemma \ref{lem:tang} and the assumption that $p$ is $n$-regular, $T_pX=T_p^gX$ is a strictly convex Banach space of dimension $n$.
Therefore the exponential map shows that $p$ is an $n$-manifold point.
\end{proof}

By the second half of Proposition \ref{prop:kkkreg}, the ``if'' part of Corollary \ref{cor:reg} is slightly generalized as follows.
We will use it in the proofs of Theorems \ref{thm:conv} and \ref{thm:noncol}.

\begin{lem}\label{lem:reg}
Let $X$ be a locally doubling, locally Busemann space without branching geodesics.
Suppose for $p\in X$, $T_pX$ is unique and isometric to a Banach space of dimension $n$ (not necessarily strictly convex).
Then $p$ is an $n$-manifold point.
\end{lem}

\subsection{Facts on Busemann with MCP}\label{sec:factbusemcp}

Next we recall two facts on Busemann spaces satisfying MCP condition (possibly collapsed).
Both are due to Kapovitch--Ketterer \cite{KK20} and Kapovitch--Kell--Ketterer \cite{KKK22}.
Although the original statements deal with CAT spaces with CD conditions, their proofs only use the Busemann convexity and the MCP condition.

The first is the non-branching property of Busemann spaces with MCP, proved in \cite[Proposition 6.9]{KK20} (see also \cite[Remark 6.10]{KK20}).

\begin{prop}\label{prop:kkbran}
Let $X$ be a locally Busemann space satisfying local {\rm MCP($K,N$)} for some measure, where $K\le 0$ and $N\ge 1$.
Then $X$ is non-branching.
\end{prop}

The second is the continuity of the tangent cone, proved in \cite[Theorem 7.1]{KKK22}.
This plays an important role in the proof of Theorem \ref{thm:mfd}.
For a shortest path $\gamma$ in a geodesic space and $\epsilon>0$, the \textit{$\epsilon$-interior} of $\gamma$ is the image of $\gamma$ with the $\epsilon$-neighborhoods of its endpoints removed.

\begin{prop}\label{prop:kkkconti}
Let $X$ be a locally Busemann space satisfying local {\rm MCP($K,N$)} for some measure, where $K\le 0$ and $N\ge 1$.
Let $\gamma$ be a shortest path in $X$.
Then the same scale Gromov--Hausdorff tangent cone (if exists) is continuous in the interior of $\gamma$ with respect to the pointed Gromov--Hausdorff convergence.

More precisely, for any $\epsilon>0$, there exists $\delta>0$ such that the following holds:
for any $p,q\in X$ with $d(p,q)<\delta$ lying in the $\epsilon$-interior of $\gamma$, we have
\[d_\GH((r^{-1}B(p,r),p),(r^{-1}B(q,r),q))<\epsilon\]
for any $0<r<\delta$, where $d_\GH$ denotes the Gromov--Hausdorff distance.
\end{prop}

\begin{rem}
Deng \cite{Den25} recently proved Propositions \ref{prop:kkbran} and \ref{prop:kkkconti} for general RCD spaces (without upper curvature bounds).
On the other hand, for general MCP/CD spaces, such nice results cannot be expected; see, e.g., \cite{KR15, Ma23}.
\end{rem}

\subsection{Consequence}

Here is an immediate consequence of the above results.
For the proof, see Theorem \ref{thm:and}, Remark \ref{rem:and}, Corollary \ref{cor:reg}, and Proposition \ref{prop:kkbran}.
See also Proposition \ref{prop:dim} for the dimension estimate.

\begin{cor}\label{cor:kkbran}
Let $X$ be a locally geodesically complete, locally Busemann space satisfying local {\rm MCP($K,N$)} for some measure, where $K\le 0$ and $N\ge 1$.
Then there exists $n\le N$ such that $X$ is a topological $n$-manifold (without boundary) and every point is an $n$-regular point.
\end{cor}

\noindent This is the geodesically complete (and possibly collapsed) version of Theorem \ref{thm:mfd}.
However, if geodesic completeness is removed, some improvements are required, which we will deal with in Theorem \ref{thm:mfd} (and Corollary \ref{cor:mfdcol}).

\section{Flatness under geodesic completeness}\label{sec:main}

In this section, we prove Theorem \ref{thm:main}.
We show that our Busemann space is of cone-type to apply Theorem \ref{thm:and}.
The proof combines a few elementary arguments based on the Busemann convexity \eqref{eq:buse}, the measure contraction property \eqref{eq:mcp}, and geodesic completeness.
Some of them are reminiscent of the proofs of Propositions \ref{prop:kkbran} and \ref{prop:kkkconti} by Kapovitch--Ketterer \cite{KK20} and Kapovitch--Kell--Ketterer \cite{KKK22}.

We first show the homogeneity of the reference measure.
Compare also with \cite[Corollary 6.5]{KKK22}, which can be viewed as the infinitesimal version of the following argument.

\begin{lem}\label{lem:hom}
Let $X$ and $m$ be as in Theorem \ref{thm:main}.
Then for any $x,y\in X$ and $r>0$, we have
\[m(B(x,r))=m(B(y,r)).\]
\end{lem}

\begin{proof}
By geodesic completeness, we can extend the shortest path $xy$ beyond both $x$ and $y$.
Let $x'$ and $y'$ be points on that extension beyond $x$ and $y$, respectively (and far away from them).
See Figure \ref{fig:hom}.

\begin{figure}[ht]
\centering
\begin{tikzpicture}
\coordinate[label=below:$x$](x)at(-1.5,0);
\coordinate[label=below:$y$](y)at(1.5,0);
\coordinate[label=below:$x'$](x')at(-4.5,0);
\coordinate[label=below:$y'$](y')at(4.5,0);

\draw(x)circle[radius=1];
\draw(y)circle[radius=1];
\draw[dashed](y)circle[radius=0.5];

\draw(-5,0)to(5,0);
\draw[dashed](-1.5,1)to(y');
\draw[dashed](-1.5,-1)to(y');

\fill(x)circle(1.5pt);
\fill(y)circle(1.5pt);
\fill(x')circle(1.5pt);
\fill(y')circle(1.5pt);
\end{tikzpicture}
\caption{}\label{fig:hom}
\end{figure}

We show $m(B(x,r))\le m(B(y,r))$.
Let $\Phi$ be the $(1-d(x,y)/d(x,y'))$-contraction map centered at $y'$ (in particular, $\Phi(x)=y$).
The measure contraction property \eqref{eq:mcp} implies
\begin{equation}\label{eq:hom1}
\left(1-\frac{d(x,y)}{d(x,y')}\right)^Nm(B(x,r))\le m(\Phi(B(x,r))).
\end{equation}
By the Busemann convexity \eqref{eq:buse}, $\Phi$ is always $1$-Lipschitz, and hence
\begin{equation}\label{eq:hom2}
\Phi(B(x,r))\subset B(y,r).
\end{equation}
Combining \eqref{eq:hom1} and \eqref{eq:hom2} and taking $d(x,y')\to\infty$, we obtain $m(B(x,r))\le m(B(y,r))$, as desired.
The symmetric procedure using $x'$ gives the opposite inequality, which together implies the desired equality.
\end{proof}

Next, using Lemma \ref{lem:hom}, we construct a ``parallel translation'' for our Busemann space.
A \textit{ray} in a metric space is an isometric embedding of $[0,\infty)$.

\begin{lem}\label{lem:trans}
Let $X$ be as in Theorem \ref{thm:main}.
Let $x,y\in X$ and $\gamma$ a ray starting at $x$.
Then there exists a ray $\eta$ starting at $y$ such that
\[d(\gamma(t),\eta(t))=d(x,y)\]
for all $t\in[0,\infty)$.
\end{lem}

\begin{proof}
Let $q_i$ be a sequence on $\gamma$ such that $d(x,q_i)\to\infty$.
By the Arzel\`a--Ascoli theorem, after passing to a subsequence, we may assume that the shortest paths $yq_i$ converge to a ray $\eta$ starting at $y$ (see Figure \ref{fig:trans}).
In what follows, we use the contraction map centered at the ``limit point'' of $\gamma$ to obtain the desired inequality.

\begin{figure}[ht]
\centering
\begin{tikzpicture}
\coordinate[label=below:$x$](x)at(0,0);
\coordinate[label=above:$y$](y)at(1,2);
\coordinate[label=right:$\gamma$](gamma)at(10,0);
\coordinate[label=right:$\eta$](eta)at(10,2);
\coordinate[label=below:$\gamma(t)$](gammat)at(2,0);
\coordinate[label=above:$\eta(t)$](etat)at(3,2);
\coordinate[label=below:$q_i$](qi)at(8,0);
\coordinate[label=below left:$h_i(t)$](hit)at($(y)!0.25!(qi)$);

\draw(x)to(gamma);
\draw(y)to(eta);
\draw(y)to(qi);
\draw[dashed](x)to(y);
\draw[dashed](gammat)to(hit);

\fill(x)circle(1.5pt);
\fill(y)circle(1.5pt);
\fill(gammat)circle(1.5pt);
\fill(etat)circle(1.5pt);
\fill(qi)circle(1.5pt);
\fill(hit)circle(1.5pt);
\end{tikzpicture}
\caption{}\label{fig:trans}
\end{figure}

We first show
\begin{equation}\label{eq:trans1}
d(\gamma(t),\eta(t))\le d(x,y).
\end{equation}
Let $h_i(t)$ be the point on the shortest path $q_iy$ such that \[\frac{d(q_i,h_i(t))}{d(q_i,y)}=\frac{d(q_i,\gamma(t))}{d(q_i,x)}.\]
The Busemann convexity \eqref{eq:buse} centered at $q_i$ implies
\begin{equation}\label{eq:trans1-1}
d(\gamma(t),h_i(t))\le d(x,y).
\end{equation}
Furthermore, since $d(q_i,y)/d(q_i,x)\to1$ as $i\to\infty$, we see that $h_i(t)$ converges to $\eta(t)$.
Taking $i\to\infty$ in \eqref{eq:trans1-1}, we obtain \eqref{eq:trans1}.

Let $D:=d(x,y)/2$.
Next we show
\begin{equation}\label{eq:trans2}
m(B(\gamma(t),D)\cup B(\eta(t),D))\ge m(B(x,D)\cup B(y,D)).
\end{equation}
Let $\lambda_i:=1-t/d(q_i,x)$ and $\Phi_i$ the $\lambda_i$-contraction map centered at $q_i$ (in particular, $\Phi_i(x)=\gamma(t)$ and $\Phi_i(y)=h_i(t)$).
Fix $0<D'<D$.
Then the measure contraction property \eqref{eq:mcp} implies
\begin{equation}\label{eq:trans2-1}
m(\Phi_i(B(x,D')\cup B(y,D')))\ge\lambda_i^Nm(B(x,D')\cup B(y,D')).
\end{equation}
By the Busemann convexity \eqref{eq:buse} and the fact that $h_i(t)\to \eta(t)$, we have
\begin{equation}\label{eq:trans2-2}
\begin{aligned}
\Phi_i(B(x,D')\cup B(y,D'))&\subset B(\gamma(t),D')\cup B(h_i(t),D')\\
&\subset B(\gamma(t),D)\cup B(\eta(t),D),
\end{aligned}
\end{equation}
provided $i$ is sufficiently large.
Combining \eqref{eq:trans2-1} and \eqref{eq:trans2-2} and taking $i\to\infty$ and $D'\to D$, we obtain \eqref{eq:trans2}.

Finally, suppose the strict inequality holds in \eqref{eq:trans1}, i.e., $d(\gamma(t),\eta(t))<d(x,y)$.
Then $B(\gamma(t),D)$ and $B(\eta(t),D)$ intersect, whereas $B(x,D)$ and $B(y,D)$ do not intersect.
Since the measures of $D$-balls are equal by Lemma \ref{lem:hom}, we have
\[m(B(\gamma(t),D)\cup B(\eta(t),D))<m(B(x,D)\cup B(y,D)).\]
This is a contradiction to \eqref{eq:trans2}.
\end{proof}

Finally, using Lemma \ref{lem:trans}, we arrive at the desired cone-type property.

\begin{prop}\label{prop:cone}
Let $X$ be as in Theorem \ref{thm:main} and $p\in X$.
Then for any $x,y\in X$ and $0\le t\le 1$, we have
\[d(x_t,y_t)=td(x,y),\]
where $x_t,y_t$ denote the $t$-intermediate points for $x,y$ with respect to $p$, respectively.
\end{prop}

\begin{proof}
Let $\gamma$ be a ray that extends the shortest path $xp$ beyond $p$.
Let $\eta$ be a ray starting at $y$ that is ``parallel'' to $\gamma$ in the sense of Lemma \ref{lem:trans}.
Set
\[z:=\eta(d(p,x)),\quad w:=\eta((1-t)d(p,x)).\]
Note that $w$ is the $(1-t)$-intermediate point for $z$ with respect to $y$ (and that $y_t$ is also the $(1-t)$-intermediate point for $p$ with respect to $y$).
See Figure \ref{fig:cone}.

\begin{figure}[ht]
\centering
\begin{tikzpicture}
\coordinate[label=below:$p$](p)at(0,0);
\coordinate[label=below:$x$](x)at(-3,0);
\coordinate[label=above:$y$](y)at(-1.5,2);
\coordinate[label=above:$z$](z)at(1.5,2);
\coordinate[label=below:$x_t$](xt)at($(x)!0.5!(p)$);
\coordinate[label=left:$y_t$](yt)at($(y)!0.5!(p)$);
\coordinate[label=above:$w$](w)at($(y)!0.5!(z)$);
\coordinate[label=right:$\gamma$](gamma)at(4,0);
\coordinate[label=right:$\eta$](eta)at(4,2);

\draw(x)to(gamma);
\draw(y)to(eta);
\draw(y)to(p);
\draw[dashed](x)to(y);
\draw[dashed](p)to(z);
\draw[dashed](xt)to(w);

\fill(p)circle(1.5pt);
\fill(x)circle(1.5pt);
\fill(y)circle(1.5pt);
\fill(z)circle(1.5pt);
\fill(xt)circle(1.5pt);
\fill(yt)circle(1.5pt);
\fill(w)circle(1.5pt);
\end{tikzpicture}
\caption{}\label{fig:cone}
\end{figure}

By Lemma \ref{lem:trans}, we have
\begin{equation}\label{eq:cone1}
d(x,y)=d(p,z)=d(x_t,w).
\end{equation}
On the other hand, by the triangle inequality and the Busemann convexity \eqref{eq:buse},
\begin{equation}\label{eq:cone2}
d(x_t,w)\le d(x_t,y_t)+d(y_t,w)\le td(x,y)+(1-t)d(p,z).
\end{equation}
Combining \eqref{eq:cone1} and \eqref{eq:cone2}, we obtain the desired equality.
\end{proof}

We are now in a position to prove Theorem \ref{thm:main}.

\begin{proof}[Proof of Theorem \ref{thm:main}]
Let $X$ be as in Theorem \ref{thm:main}.
By Proposition \ref{prop:cone}, $X$ is of cone-type (at every point) and in particular non-branching (cf.\ Proposition \ref{prop:kkbran}).
Therefore, Theorem \ref{thm:and} implies that $X$ is isometric to a strictly convex Banach space of dimension $n$ for some $n$.
Since the Hausdorff dimension is less than or equal to the dimension parameter $N$, \cite[Corollary 2.7]{Oh07}, we have $n\le N$.
Furthermore, Lemma \ref{lem:hom} shows that the measure $m$ is translation invariant, and hence it is a constant multiple of the $n$-dimensional Hausdorff (or Lebesgue) measure of $X$.
This completes the proof.
\end{proof}

We can also provide an alternative proof using the idea from sub-Finsler geometry (\cite{Le11, Be18} cf.\ \cite{Le15, HaL23}).
We refer the reader to \cite{Le11} or \cite[Section 2.2]{MMR25} for the definition of a sub-Finsler Carnot group.
The authors thank Enrico Le Donne for the proof of Lemma \ref{lem:sub}.

\begin{thm}[{\cite[Theorem 1.2]{Le11}}]\label{thm:le}
Let $(X,d,m)$ be a metric measure space with a doubling measure $m$.
Suppose for $m$-a.e.\ $p\in X$, the Gromov--Hausdorff tangent cone $T_pX$ is unique.
Then $T_pX$ is isometric to a sub-Finsler Carnot group for $m$-a.e.\ $p\in X$.
\end{thm}

\begin{lem}\label{lem:sub}
Any sub-Finsler Carnot group satisfying the Busemann convexity is a strictly convex Banach space. 
\end{lem}

\noindent Lemma \ref{lem:sub} is an immediate consequence of the following fact shown in \cite[Proposition 3.2]{Be18} (cf.\ \cite[Proof of Proposition 6.1]{HaL23}): any non-Banach sub-Finsler Carnot group is not uniquely geodesic.

\begin{proof}[Alternative proof of Theorem \ref{thm:main}]
Let $X$ be as in Theorem \ref{thm:main}.
By Lemma \ref{lem:tang}, the Gromov--Hausdorff tangent cone of $X$ is uniquely defined at every point.
Therefore, by Theorem \ref{thm:le}, there exists $p\in X$ such that $T_pX$ is isometric to a sub-Finsler Carnot group.
By Proposition \ref{prop:cone}, $T_pX$ is isometric to $X$ via the exponential map (see Section \ref{sec:tang}).
In particular, $T_pX$ is Busemann convex.
Since any Busemann convex sub-Finsler Carnot group is a strictly convex Banach space (Lemma \ref{lem:sub}), this completes the proof.
(Alternatively, one can also use the fact that any Busemann concave sub-Finsler Carnot group is a strictly convex Banach space, \cite[Proposition 2.5]{Ke19}).
\end{proof}

\begin{rem}
Here is yet another proof of Theorem \ref{thm:main} using the theory of G-space developed in \cite{Bu55}.
Note that our Busemann space $X$ of Theorem \ref{thm:main} is a straight G-space in the sense of \cite[p.\ 38]{Bu55} (cf.\ Proposition \ref{prop:kkbran}).
Since we have shown that $X$ is of cone-type at every point in Proposition \ref{prop:cone}, $X$ has curvature $0$ in the sense of \cite[(36.2)]{Bu55}.
Therefore, $X$ is isometric to a finite-dimensional Banach space, by the proof of \cite[(39.12)]{Bu55}.
\end{rem}

\begin{rem}
As implicitly used in the above proofs, the MCP condition includes that $X$ is a metric measure space with respect to the reference measure, i.e., $m$ is locally positive and locally finite on $X$.
It would be natural to ask if any locally compact, geodesically complete Busemann space admits a ``canonical measure'' defined by its metric structure so that it becomes a metric measure space.
In fact, such a canonical measure, consisting of the Hausdorff measures of different dimensions, exists in the CAT setting; see \cite[Theorem 1.4]{LN19}.
However, the current structure theory of geodesically complete Busemann spaces developed in \cite{FG25} is still insufficient to construct such a measure (as it focused on the topological aspects).
\end{rem}

Finally, we give an example showing that Theorem \ref{thm:main} does not hold for CD($0,\infty$) (note that there is no such notion as MCP($0,\infty$)).

\begin{ex}\label{ex:gauss}
    Let $\mathbb{H}^2$ be the hyperbolic plane of constant sectional curvature $\equiv -1$ equipped with the radial coordinates: $g=dr^2 + \sinh^2(r) d\theta^2$.
On this space, we consider a Gaussian-type measure $e^{-V}\mathrm{vol}$, where $V(r,\theta):=r^2$ denotes the squared distance from the base point.
By a direct computation, the Hessian of the square distance function $V$ is
\begin{align*}
    \mathrm{Hess}(V)&{}=2dr^2 + 2r\cosh(r)\sinh(r)d\theta^2\\
    &{}=2g+2\sinh(r)\left(r\cosh(r)-\sinh(r)\right)d\theta^2\geq 2g.
\end{align*}
Therefore we have
$\Ric_{\infty,V}:=\Ric + \mathrm{Hess}(V)\geq g$,
which implies the CD($1,\infty$) condition by \cite[Theorem 17.37]{Vil}.
\end{ex}

\section{Topological regularity without geodesic completeness}\label{sec:mfd}

In this section, we prove Theorem \ref{thm:mfd}.
First, in Section \ref{sec:col}, we prove general results for locally Busemann spaces satisfying the local MCP condition (possibly collapsed).
After that, in Section \ref{sec:noncol}, we discuss the non-collapsing case and prove Theorem \ref{thm:mfd}.
Note that all results in Section \ref{sec:col} assume the existence of a manifold point and the non-collapsing assumption in Section \ref{sec:noncol} is used only to show the existence of a manifold point
(Theorem \ref{thm:noncol}).

\subsection{General case (possibly collapsing)}\label{sec:col}
In this subsection, we prove general results for locally Busemann spaces satisfying the local MCP condition, \textit{assuming the existence of a manifold point (but without assuming the non-collapsing condition)}.
We first establish the convexity of the set of manifold points (Theorem \ref{thm:conv}).
We then show the equivalence of several different notions of inner points (Theorem \ref{thm:inner}).
Finally, we observe that the set of non-manifold points is indeed the manifold boundary (Corollary \ref{cor:bdry}).
These are Busemann+MCP versions of the results of Kapovitch--Kell--Ketterer \cite[Sections 3--4]{KKK22} for CAT+CD spaces.
However, the proofs are somewhat different, and in particular, the existence of a manifold point in \cite[Theorem 3.15]{KKK22} is not shown in this subsection.

Before we get into the above discussion, we show the following basic proposition generalizing part of Theorem \ref{thm:main}.

\begin{prop}\label{prop:dim}
Let $X$ be a locally Busemann space satisfying local {\rm MCP($K,N$)} for some measure, where $K\le0$ and $N\ge 1$.
If $X$ contains an $n$-manifold point, then $n\le N$.
\end{prop}

\begin{proof}
Let $p$ be an $n$-manifold point of $X$.
By Corollary \ref{cor:reg}, $p$ is an $n$-regular point, i.e., $T_pX$ is unique and isometric to a strictly convex Banach space of dimension $n$.
Since $X$ satisfies the local MCP($K,N$), $T_pX$ satisfies the MCP($0,N$) for some rescaled limit measure.
Hence the claim follows from the fact that the Hausdorff dimension is less than or equal to the dimension parameter, \cite[Corollary 2.7]{Oh07}.
\end{proof}

Now let us get into the main topic.
The key observation in this subsection is the following strong convexity of the set of manifold points.

Let $X$ be a geodesic space.
For a positive integer $n$, we denote by $\In_nX$ the set of $n$-manifold points in $X$.
For a shortest path $\gamma$ in $X$, we denote by $\In\gamma$ the relative interior of $\gamma$, that is, the image of $\gamma$ with its endpoints removed.

\begin{thm}\label{thm:conv}
Let $X$ be a locally Busemann space satisfying local {\rm MCP($K,N$)} for some measure, where $K\le0$ and $N\ge 1$.
Then for any positive integer $n$, the $n$-manifold part $\In_nX$ (if nonempty) is strongly convex in $X$ in the following sense: if a shortest path $\gamma$ intersects $\In_nX$, then its interior $\In\gamma$ is contained in $\In_nX$.
\end{thm}

\noindent This theorem is a generalization of \cite[Theorem 3.19(iii)]{KKK22} in the CAT+CD case, but it seems difficult to adapt their original proof to our setting.
This is because, as mentioned in Section \ref{sec:out}, Busemann spaces do not have rich structures as CAT spaces.
For example, in \cite[Propositions 3.12, 3.13]{KKK22}, the original proof uses the splitting theorem for CAT($0$) spaces, as well as the standard fact that the tangent cone of a CAT($\kappa$) space is a CAT($0$) space.

To overcome this issue, we shall take a different approach suggested in \cite[Remark 7.3]{KKK22}.
The key is the continuity of the tangent cone, Proposition \ref{prop:kkkconti} (\cite[Theorem 7.1]{KKK22}), and the following basic lemma on Banach spaces.

\begin{lem}\label{lem:ban}
Let $(X_i,o_i)$ be a sequence of pointed metric spaces, each of which is isometric to a Banach space of dimension $n$.
Suppose $(X_i,o_i)$ converges to $(X,o)$ in the pointed Gromov--Hausdorff topology.
Then $X$ is also isometric to a Banach space of dimension $n$.
\end{lem}

\begin{proof}
Let $f_i:X_i\to (\R^n,|\cdot|_i)$ be an isometry, where $|\cdot|_i$ is a norm on $\R^n$.
Let $B_i(r)$ and $B(r)$ denote the $r$-balls centered at the origin of $\R^n$, with respect to $|\cdot|_i$ and the standard Euclidean norm, respectively.
By choosing $f_i$ and $|\cdot|_i$ appropriately, we may assume that
\begin{equation}\label{eq:ban}
B(1)\subset B_i(1)\subset B(n).
\end{equation}
Indeed, by the theorem of John \cite{Jo}, there exists an ellipsoid  $E_i\subset\R^n$ such that
\[E_i\subset B_i(1)\subset nE_i.\]
Choosing a linear isomorphism $A:\R^n\to\R^n$ that takes $E_i$ to $B(1)$ and replacing $f_i$ and $|\cdot|_i$ with $f_i':=Af_i$ and $|\cdot|_i':=|A^{-1}\cdot|_i$, respectively, we obtain the desired property \eqref{eq:ban}.

Consider the norm $|\cdot|_i$ as a function on the standard unit sphere $S(1)\subset\R^n$.
By the property \eqref{eq:ban} and the Arzel\`a--Ascoli theorem, we may assume that $|\cdot|_i$ converges to a norm $|\cdot|$ on $\R^n$.
Since $(X_i,o_i)$ converges to $(X,o)$, we see that $X$ is isometric to $(\R^n,|\cdot|)$.
\end{proof}

\begin{proof}[Proof of Theorem \ref{thm:conv}]
Let $\gamma$ be as in Theorem \ref{thm:conv}.
Since $\In_nX$ is open in $X$, the intersection $\In\gamma\cap\In_nX$ is open in $\In\gamma$.
Since $\In\gamma$ is connected, it suffices to prove that $\In\gamma\cap\In_nX$ is closed in $\In\gamma$.

Suppose $p_i\in\In\gamma\cap\In_nX$ converges to $p\in\In\gamma$.
We prove that $p\in\In_nX$.
We first show that $T_pX$ is unique and is isometric to the pointed Gromov--Hausdorff limit of $T_{p_i}X$.
Let $o_i$ denote the apex of $T_{p_i}X$.
By the triangle inequality for the Gromov--Hausdorff distance, we have
\begin{align*}
&d_\GH((r^{-1}B(p,r),p),(B(o_i,1),o_i))\\
&\le d_\GH((r^{-1}B(p,r),p),(r^{-1}B(p_i,r),p_i))+d_\GH((r^{-1}B(p_i,r),p_i),(B(o_i,1),o_i)).
\end{align*}
Fix an arbitrarily small $\epsilon>0$.
By the assumption $p\in\In\gamma$ and Proposition \ref{prop:kkkconti}, the first term in the right-hand side above is less than $\epsilon$ for any sufficiently small $r$ and any sufficiently large $i$.
The second term is also less than $\epsilon$, provided $r$ is sufficiently small for each fixed $i$.
Therefore,
\begin{equation}\label{eq:conv}
d_\GH((r^{-1}B(p,r),p),(B(o_i,1),o_i))<2\epsilon
\end{equation}
for any large $i$ and any small $r$ depending on $i$.
By the triangle inequality again, we have
\[d_\GH((r^{-1}B(p,r),p),(s^{-1}B(p,s),p))<4\epsilon\]
for any sufficiently small $r,s$.
Since $\epsilon$ is arbitrary, this implies that $T_pX$ is unique.
Taking $r\to0$ and $i\to\infty$ in \eqref{eq:conv}, we see that $T_{p_i}X$ converges to $T_pX$.

By Corollary \ref{cor:reg}, $T_{p_i}X$ is isometric to a Banach space of dimension $n$.
Therefore, by Lemma \ref{lem:ban}, $T_pX$ is also isometric to a Banach space of dimension $n$.
By Lemma \ref{lem:reg}, we see that $p$ is an $n$-manifold point.
This completes the proof.
\end{proof}

\begin{rem}

In the CAT case \cite[Remark 7.3]{KKK22}, we did not need Lemma \ref{lem:ban}, since the tangent cone at a regular point of a CAT space is isometric to Euclidean space, which has a unique norm.
However, for Busemann spaces, the tangent norm at a regular point may vary from point to point, and thus we need Lemma \ref{lem:ban}.
See also Question \ref{ques:ber} towards the homogeneity of tangent norms.
\end{rem}

The following corollary immediately follows from Theorem \ref{thm:conv}.

\begin{cor}\label{cor:conv}
Let $X$ be a locally Busemann space with local {\rm MCP($K,N$)} for some measure, where $K\le0$ and $N\ge 1$.
If $\In_nX$ is nonempty, then such $n$ is unique and $\In_nX$ is dense in $X$.
\end{cor}

\begin{proof}
First, suppose $\In_nX$ and $\In_{n'}X$ are nonempty.
Take $p\in\In_nX$ and $p'\in\In_{n'}X$ and connect them by a shortest path $\gamma$.
Since $\In\gamma$ intersects both $\In_nX$ and $\In_{n'}X$, applying Theorem \ref{thm:conv}, we see that $n=n'$.

Next, suppose $\In_nX$ is nonempty and let $x\in X$ be an arbitrary point.
Take $p\in\In_nX$ and connect $p$ to $x$ by a shortest path $\gamma$.
Applying Theorem \ref{thm:conv}, we see that every interior point of $\gamma$ is an $n$-manifold point, which converges to $x$.
\end{proof}

\begin{rem}
The constancy of dimension as above does not hold for general MCP (even CD) spaces, see \cite{KR15, Ma23}.
On the other hand, it does hold for RCD spaces, see \cite{BS20}.
\end{rem}

Using Theorem \ref{thm:conv}, we show the equivalence of several notions of inner points introduced in Section \ref{sec:factbuse}, assuming the existence of a manifold point.
Compare the following with \cite[Proposition 3.13]{KKK22}.

\begin{thm}\label{thm:inner}
Let $X$ be a locally Busemann space with local {\rm MCP($K,N$)} for some measure, where $K\le0$ and $N\ge 1$.
Suppose $X$ contains an $n$-manifold point $x_0$.
Then for any $p\in X$, the following conditions are equivalent.
\begin{enumerate}
\item $p$ is an $n$-manifold point;
\item $p$ is an $n$-regular point;
\item $p$ is a topologically inner point;
\item $p$ is a geometrically inner point;
\item there exists a geodesic from $x_0$ to $p$ that is extendable beyond $p$.
\end{enumerate}
\end{thm}

\begin{proof}

The following diagram summarizes all the implications discussed below, indicating the conditions used.
In particular, the existence of a manifold point is only used to show the implication from (4) to (1) passing through (5).
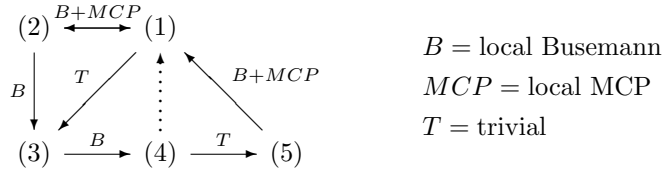
\[
\begin{diagram}
 \node{(2)}\arrow{e,t,2}{B+MCP}\arrow{s,l,2}{B}
 \node{(1)}\arrow{w,d,2}{}\arrow{sw,t,2}{T}\\
 \node{(3)}\arrow{e,t,2}{B}
 \node{(4)}\arrow{e,t,2}{T}\arrow{n,..}
 \node{(5)}\arrow{nw,t,2}{B+MCP}
\end{diagram}
\qquad\qquad
\begin{aligned}
   & B=\text{local Busemann}\\
   & MCP=\text{local MCP}\\
   & T=\text{trivial}
\end{aligned}
\]
The equivalence between (1) and (2) was already proved in Corollary \ref{cor:reg} (note that the non-branching assumption, as well as the local doubling condition, follows from the local MCP; see Proposition \ref{prop:kkbran}).
We also showed that (2) implies (3) in Proposition \ref{prop:kkkreg} (and clearly (1) implies (3)).
By Proposition \ref{prop:ls}, (3) implies (4).
Clearly (4) implies (5).
Finally, by Theorem \ref{thm:conv}, (5) implies (1).
\end{proof}

\begin{rem}
In the CAT case, the circle of equivalence is closed by showing (4) to (2) directly without passing through (5), using the splitting theorem and the fact that $X$ is non-branching.
See \cite[Corollary 3.10]{KKK22}.
In particular, Theorem \ref{thm:inner} in the CAT case does not require the existence of a manifold point.
\end{rem}

Now we discuss the boundary of a Busemann space with MCP.
Our treatment is different from Kapovitch--Kell--Ketterer \cite[Section 4]{KKK22}, where they introduced the \textit{geometric boundary} and proved that it coincides with the manifold boundary.
Here we define it more directly as the complement of the set of manifold points.

Let $X$ be a locally Busemann space with local MCP($K,N$).
With Corollary \ref{cor:conv} in mind, we simply denote by $\In X$ the set of manifold points in $X$.
We say that $p\in X$ is a \textit{boundary point} if it is not a manifold point, and denote by $\partial X$ the set of boundary points, i.e., $\partial X=X\setminus\In X$.
Note that Theorem \ref{thm:inner} gives several equivalent definitions of a boundary point, provided $\In X\neq\emptyset$.

We first show the vanishing of the reference measure on the boundary.

\begin{cor}\label{cor:meas}
Let $X$ be a locally Busemann space with local {\rm MCP($K,N$)} for some measure, where $K\le0$ and $N\ge 1$.
Suppose $\In X$ is nonempty.
Then $\In X$ has full measure in $X$, i.e., $\partial X$ has measure zero.
\end{cor}

\begin{proof}
Recall that the non-extendable set in an MCP space has measure zero (\cite[Section 3.3]{vo08}).
Here the \textit{non-extendable set} for a point $x_0$ in a geodesic space is the set of points $p$ such that any shortest path $x_0p$ cannot be extended beyond $p$ as a shortest path.

For any $x\in X$, we take $r>0$ such that the closed ball $\bar B(x,r)$ is convex and is a Busemann and MCP($K,N$) space with respect to the restricted metric and measure.
By Corollary \ref{cor:conv}, there exists a manifold point $x_0$ in the interior $B(x,r)$.
From the equivalence between (1) and (5) of Theorem \ref{thm:inner}, we see that the set of manifold points in $\bar B(x,r)$ coincides with the extendable set for $x_0$ in $\bar B(x,r)$ (note that by the Busemann condition, every geodesic contained in $\bar B(x,r)$ is a shortest path).
Therefore, by the fact in the previous paragraph, $\In X$ has full measure in $B(x,r)$.
Since $X$ is proper, a covering argument shows the claim.
\end{proof}

Next we see that the boundary defined above is indeed the manifold boundary, as proved in \cite[Theorem 4.3]{KKK22}.

\begin{cor}\label{cor:bdry}
Let $X$ be a locally Busemann space with local {\rm MCP($K,N$)} for some measure, where $K\le0$ and $N\ge 1$.
Suppose $X$ contains an $n$-manifold point and let $p\in\partial X$.
Then there exists a neighborhood of $p$ that is homeomorphic to the half space $\R^n_+$.
\end{cor}

\begin{proof}
The proof is exactly the same as that of \cite[Theorem 4.3]{KKK22}.
In fact, the original proof is based only on a few properties of a metric space belonging to the class $\mathcal C$, which are listed in \cite[Proposition 3.13, Theorem 3.19]{KKK22}.
We have already established those properties in our Busemann+MCP setting, under the additional assumption of the existence of a manifold point.
Nevertheless, for the convenience of the reader, we include a complete argument.

In what follows, we restrict our attention to a convex neighborhood of $p$ in $X$ that satisfies the Busemann condition and MCP (in particular, every geodesic in it is a shortest path).
By Corollary \ref{cor:conv} and Theorem \ref{thm:inner}, we can find a regular point $q$ near $p$.
Then there exists a small closed ball $\bar B(q,r)$, where $r<d(p,q)$, that is homeomorphic to a closed $n$-disk via the exponential map; see Section \ref{sec:tang} and Proposition \ref{prop:kkbran}.
Suppose the boundary $\partial B(q,r)$ intersects the shortest path $qp$ at a point $p'$.
Let $U$ be a small neighborhood of $p'$ in $\partial B(q,r)$ that is homeomorphic to $\R^{n-1}$.

By Proposition \ref{prop:kkbran} and Theorem \ref{thm:inner}, we see that for any $x\in U$, there exists a unique maximal shortest path emanating from $q$, passing through $x$, and reaching $\partial X$ (note that such a maximal shortest path certainly reaches $\partial X$ since $p\in\partial X$).
Let $\Psi(x)\in\partial X$ denote its endpoint.
We shall observe that the map
\[\Psi:U\to\Psi(U)\subset \partial X\]
is a homeomorphism to an open neighborhood of $p$ in $\partial X$.
The injectivity of $\Psi$ follows from the uniqueness of shortest paths.
The continuity of $\Psi^{-1}$ and the openness of $\Psi(U)$ follow from the continuous dependence of shortest paths on their endpoints.
We will verify that $\Psi$ is continuous.
Suppose $x_i\in U$ converges to $x\in U$.
Then, after passing to a subsequence, we may assume that the maximal shortest path $q\Psi(x_i)$ converges to a shortest path $qy$ passing through $x$.
Since $q\Psi(x)$ is the unique maximal shortest path passing through $x$, $qy$ is contained in $q\Psi(x)$.
If it is strictly contained, then Theorem \ref{thm:inner} implies $y\in\In X$. Since $\In X$ is open, we have $\Psi(x_i)\in\In X$ for any large $i$.
By Theorem \ref{thm:inner} again, we see that $q\Psi(x_i)$ is extendable beyond $\Psi(x_i)$, a contradiction.
Therefore, $y=\Psi(x)$, as desired.

Now we define a map $\Phi:U\times(1/2,1]\to X$ by
\[(x,t)\mapsto\text{the $t$-intermediate point for $\Psi(x)$ with respect to $q$.}\]
The above argument implies that $\Phi$ is a homeomorphism to an open neighborhood of $p$ in $X$.
Since $U$ is homeomorphic to $\R^{n-1}$, this is the desired neighborhood.
\end{proof}

We conclude this subsection with the following corollary summarizing Proposition \ref{prop:dim}, Theorems \ref{thm:conv} and \ref{thm:inner}, and Corollaries \ref{cor:conv}, \ref{cor:meas}, and \ref{cor:bdry}.
This is a general (i.e., possibly collapsing) version of Theorem \ref{thm:mfd},
assuming the existence of a manifold point.
Compare with \cite[Theorem 1.1]{KKK22} and \cite[Corollary 1.2]{KK20}.

\begin{cor}\label{cor:mfdcol}
Let $X$ be a locally Busemann space with local {\rm MCP($K,N$)} for some measure, where $K\le0$ and $N\ge 1$.
Suppose $X$ contains an $n$-manifold point.
Then $n\le N$ and $X$ is a topological $n$-manifold with boundary.
The manifold interior of $X$ is geodesically convex, has full measure, and coincides with the set of $n$-regular points.
\end{cor}

\begin{rem}
In Corollary \ref{cor:mfdcol}, the Busemann convexity cannot be replaced with a weaker assumption of convex geodesic bicombing.
The counterexample can be found in \cite{KR15} (cf.\ \cite{Ma23}).
Indeed, the example given in the proof of \cite[Theorem 3]{KR15} satisfies MCP($0,3$) and admits a convex geodesic bicombing (by considering the concatenations of affine geodesics glued at the origin).
However, $X$ is not even a topological manifold.
This happens because the family of geodesics that satisfies the convexity is different from the one used to show the MCP.
\end{rem}

\subsection{Non-collapsing case}\label{sec:noncol}

From now on we consider the non-collapsing case and prove Theorem \ref{thm:mfd}.
In view of Corollary \ref{cor:mfdcol}, it remains to show the existence of a manifold point
under the non-collapsing assumption.

The following theorem is the main result of this subsection.

\begin{thm}\label{thm:noncol}
Let $X$ be a locally Busemann space satisfying local non-collapsed {\rm MCP($K,n$)}, where $K\le 0$ and $n\ge 1$.
Then $n$ is an integer and $X$ contains an open dense subset of $n$-manifold points with full measure.
Furthermore, $X$ is $n$-rectifiable.
\end{thm}

\noindent Note that the rectifiability also follows from the stronger claims, Theorem \ref{thm:int} and Corollary \ref{cor:meas}.
We prove Theorem \ref{thm:noncol} with the help of the recent structure result for non-collapsed MCP spaces by Magnabosco--Mondino--Rossi \cite{MMR25}.

\begin{thm}[{\cite[Theorem 1.2]{MMR25}}]\label{thm:mmr}
Let $(X,d,m)$ be an {\rm MCP($K,n$)} space, where $K\in \R$ and $n\in(1,\infty)$ (or $K\le 0$ and $n=1$).
Assume that
\begin{enumerate}
\item the lower and upper densities are positive and finite almost everywhere, i.e.,
\[0<\liminf_{r\to0}\frac{m(B(p,r))}{r^n}\le\limsup_{r\to0}\frac{m(B(p,r))}{r^n}<\infty\]
for $m$-a.e.\ $p\in X$;
\item the Gromov--Hausdorff tangent cone $T_pX$ is unique up to isometry for $m$-a.e.\ $p\in X$. 
\end{enumerate}
Then $n$ is an integer and $T_pX$ is isometric to an $n$-dimensional Banach space for almost every $p\in X$.
Furthermore, $X$ is $(m,n)$-rectifiable.
\end{thm}
\noindent The last condition, \textit{$(m,n)$-rectifiability}, means that $X$ is covered by a countable collection of Lipschitz images of Borel subsets of $\R^n$ with positive $n$-dimensional Hausdorff measure, up to an $m$-null set.
See \cite[Definition 3.2]{MMR25} for the precise definition.
Below we will consider the $m=\Hcal^n$ case,  where the $(m,n)$-rectifiability reduces to the standard $n$-rectifiability.

\begin{rem}\label{rem:density}
    Although \cite[Theorem 1.2]{MMR25} is claimed for $K\in \R$ and $n>1$, the conclusion still holds for $K\leq 0$ and $n=1$, essentially because \cite[Theorem 4.3]{MMR25} holds for $N=1$ trivially.
    Note also that the $\liminf$ and $\limsup$ in the assumption (1) are actually equal and positive by the Bishop--Gromov inequality (see Section \ref{sec:mcp}).
\end{rem}

The core of the proof of Theorem \ref{thm:noncol} is to show the uniqueness of the tangent cone of $X$ at almost every point, as assumed in Theorem \ref{thm:mmr}(2).
The conclusion then follows from Theorem \ref{thm:mmr} and our previous arguments.
In what follows, we first observe that the non-collapsed MCP assumption implies the almost extendability of geodesics at almost every point (Claim \ref{clm:noncol1}).
We then prove that this almost extendability, together with the Busemann convexity, implies the uniqueness of the tangent cone at that point (Claim \ref{clm:noncol2}).
The latter is a generalization of Lemma \ref{lem:tang}.

\begin{proof}[Proof of Theorem \ref{thm:noncol}]
Let $X$ be as in Theorem \ref{thm:noncol}.
We will apply Theorem \ref{thm:mmr} to each Busemann+MCP neighborhood in $X$.
Then, by a covering argument, we see that $X$ is $n$-rectifiable and almost every point of $X$ has a unique Banach tangent cone of dimension $n$ (note that $X$ is proper).
In particular, the existence (and denseness) of manifold points follows from Lemma \ref{lem:reg}  (cf.\ Corollary \ref{cor:conv}).

We will check that the assumptions of Theorem \ref{thm:mmr} are satisfied.
As explained in Remark \ref{rem:density}, the lower density is positive everywhere.
Since $\Hcal^n$ is locally finite, the upper density is finite almost everywhere, that is, for almost every $p\in X$ we have
\begin{equation}\label{eq:upper}
\limsup_{r\to0}\frac{\Hcal^n(B(p,r))}{r^n}<\infty
\end{equation}
(see, e.g., \cite[Theorem 3.6]{Simbook}).
Hence it remains to show the uniqueness of the Gromov--Hausdorff tangent cone $T_pX$ for almost every $p\in X$.
We will show that $T_pX$ is isometric to the geodesic tangent cone $T_p^gX$ for every point $p$ with finite upper density \eqref{eq:upper}.

Recall that Lemma \ref{lem:tang} showed the above claim under the assumption that any shortest path starting at $p$ is extendable to uniform length.
Here, we first show the ``almost extendability'' of geodesics for any point $p$ satisfying \eqref{eq:upper}, and then derive the uniqueness of the tangent cone at $p$ from this almost extendability.
More precisely, we prove the following two claims independently.

\begin{clm}\label{clm:noncol1}
Suppose the inequality \eqref{eq:upper} holds for $p\in X$. 
Then the following almost extendability of geodesics holds.
\begin{itemize}
\item[($*$)]
For any $\delta>0$, there exists $R>0$ such that for any $0<r\le R$ and $0<t<1$, the image of the $t$-contraction map
\[\Phi_t:B(p,r)\to B(p,tr)\]
centered at $p$ is $(\delta tr)$-dense in $B(p,tr)$.
\end{itemize}
\end{clm}

\begin{clm}\label{clm:noncol2}
Suppose the almost extendability of geodesics ($*$) holds for $p\in X$.
Then $T_pX$ is unique and isometric to $T_p^gX$.
\end{clm}

\noindent It should be emphasized that the proof of Claim \ref{clm:noncol1} only uses the MCP, whereas the proof of Claim \ref{clm:noncol2} only uses the Busemann convexity.
In particular, Claim \ref{clm:noncol2} generalizes Lemma \ref{lem:tang}.
Indeed, the assumption of Lemma \ref{lem:tang} that any shortest path starting at $p$ is extendable to uniform length is equivalent to that Condition ($*$) holds for $\delta=0$ and some $R>0$.

\begin{proof}[Proof of Claim \ref{clm:noncol1}]
For simplicity, we first consider the MCP($0,n$) case.
The general MCP($K,n$) case with $K<0$ only requires a minor modification, which will be explained later.

The idea of the proof is simple.
By the MCP($0,n$), the value $\Hcal^n(B(p,r))/r^n$ is non-decreasing as $r\to 0$.
If Condition ($*$) does not hold for a fixed $\delta$, then for any $R>0$, there exist $0<r\le R$ and $0<t<1$ such that this value will increase by a fixed amount (depending on $\delta$) when moving from $r$ to $tr$.
Repeating this procedure infinitely many times, we get a contradiction to \eqref{eq:upper}.
The precise argument goes as follows.

Suppose that Condition ($*$) does not hold.
Then, for some $\delta>0$, there exists monotonically decreasing sequences $r_i>s_i>r_{i+1}>0$ such that the image of the $(s_i/r_i)$-contraction map
\[\Phi_i:B(p,r_i)\to B(p,s_i)\]
is not $(\delta s_i)$-dense in $B(p,s_i)$.
Moreover, we may assume $r_i\to0$.
See Figure \ref{fig:noncol1}.

\begin{figure}[ht]
\centering
\begin{tikzpicture}
\coordinate[label=below:$p$](p)at(0,0);
\coordinate(qi)at(1.1,1.1);

\draw(p)circle[radius=3];
\draw(p)circle[radius=2];
\draw(p)circle[radius=1];

\draw[dashed](p)to(25:3);
\draw[dashed](p)to(65:3);

\fill[lightgray](qi)circle[radius={0.4}];
\node at(qi){$\frac{\delta s_i}2$};

\node[below right]at(-45:3){$r_i$};
\node[below right]at(-45:2){$s_i$};
\node[below right]at(-45:1){$r_{i+1}$};

\fill(p)circle(1.5pt);
\end{tikzpicture}
\caption{}\label{fig:noncol1}
\end{figure}

By the MCP($0,n$) condition, we have
\begin{equation}\label{eq:upper1}
\frac{\Hcal^n(\Phi_i(B(p,r_i)))}{s_i^n}\ge\frac{\Hcal^n(B(p,r_i))}{r_i^n}.
\end{equation}
By the definition of $s_i$, the image of $\Phi_i$ misses some $(\delta s_i/2)$-ball that is entirely contained in $B(p,s_i)$ (indeed, if $q_i\in B(p,s_i)$ is not contained in the $(\delta s_i)$-neighborhood of the image of $\Phi_i$, then let $q_i'$ be a point on the shortest path $pq_i$ at distance $\delta s_i/2$ from $q_i$ and consider the $(\delta s_i/2)$-ball around $q_i'$).
By the MCP($0,n$), we see that the volume of this $(\delta s_i/2)$-ball is bounded from below by $c(\delta s_i)^n$, where $c$ is a positive number independent of $i$ (indeed, compare the volume of this $(\delta s_i/2)$-ball with the volume of a concentric ball of fixed radius independent of $i$; since the center converges to $p$, the latter is uniformly bounded below by a positive number).
Therefore, we have
\begin{equation}\label{eq:upper2}
\frac{\Hcal^n(B(p,s_i))}{s_i^n}\ge\frac{\Hcal^n(\Phi_i(B(p,r_i)))}{s_i^n}+c\delta^n.
\end{equation}
Finally, by the MCP($0,n$) again, we have
\begin{equation}\label{eq:upper3}
\frac{\Hcal^n(B(p,r_{i+1}))}{r_{i+1}^n}\ge\frac{\Hcal^n(B(p,s_i))}{s_i^n}.
\end{equation}
Combining \eqref{eq:upper1}, \eqref{eq:upper2}, and \eqref{eq:upper3}, we get
\begin{equation}\label{eq:upper4}
\frac{\Hcal^n(B(p,r_{i+1}))}{r_{i+1}^n}\ge\frac{\Hcal^n(B(p,r_i))}{r_i^n}+c\delta^n.
\end{equation}
Adding the inequalities \eqref{eq:upper4} for all $i$, we arrive at a contradiction to $\eqref{eq:upper}$.

For the general MCP($K,n$) case, we use the following property instead of the MCP($0,n$): for any $\epsilon>0$, there exists $r_0>0$ such that for any $0<r\le r_0$ and $0\le t\le 1$, we have
\[\Hcal^n(\Phi_t(B(p,r)))\ge(1-\epsilon)t^n\Hcal^n(B(p,r)).\]
This is immediate from the definition of the general MCP($K,n$) (see Remark \ref{rem:mcp}) and the elementary fact that $\lim_{x\to0}\sinh(x)/x=1$ (and hence $\sinh(tx)/tx\to1$ uniformly for $0<t\le 1$ as $x\to 0$).
All the previous estimates remain valid with such small errors.
In particular, instead of \eqref{eq:upper4}, we have
\[\frac{\Hcal^n(B(p,r_{i+1}))}{r_{i+1}^n}\ge(1-\epsilon_i)\left(\frac{\Hcal^n(B(p,r_i))}{r_i^n}+c\delta^n\right),\]
where $\epsilon_i\to0$ as $r_i\to0$.
Therefore, to get a contradiction to \eqref{eq:upper} by combining these inequalities for all $i$, it suffices to choose $r_i$ so that $\epsilon_i\le 2^{-i}$ and hence $\prod_{i=1}^\infty(1-\epsilon_i)>0$.
This is possible since one may choose arbitrarily small $r_i<s_{i-1}$ at the beginning of the argument by contradiction.
\end{proof}

\begin{proof}[Proof of Claim \ref{clm:noncol2}]
The following argument is a modification of \cite[Lemma 5.6]{LN19} (cf.\ \cite[Lemma 4.5]{FG25}).
For any small $r>0$, we consider the logarithmic map $\log_p:B(p,r)\to B(o,r)\subset T_p^gX$ defined by
\[\log_p(x):=(px,d(p,x))\in\Gamma_p\times[0,\infty),\]
where $px$ denotes the unique
shortest path from $p$ to $x$ and $\Gamma_p$ is the set of shortest paths emanating from $p$ (see Section \ref{sec:tang} for the definition of $T_p^gX$).
In what follows we denote by $d^*$ the metric of $T_p^gX$.

We prove that $\log_p$ is almost distance-preserving and almost surjective in the corresponding scale.
More precisely, we show that for any $\delta>0$, there exists $r_0>0$ such that for any $0<r\le r_0$, the following two conditions hold:
\begin{enumerate}
\item for any $x,y\in B(p,r)$, we have
\[\left|d(x,y)-d^*(\log_px,\log_py)\right|\le\delta r;\]
\item for any $v\in B(o,r)$, there exists $x\in B(p,r)$ such that
\[d^*(\log_px,v)\le\delta r.\]
\end{enumerate}
Then, for any $R>0$ and sufficiently large $\lambda>0$ with $\lambda^{-1}R<r_0$, the rescaled map $\log_p:\lambda B(p,\lambda^{-1}R)\to \lambda B(o,\lambda^{-1}R)$ is a $(\delta R)$-Gromov--Hausdorff approximation.
By the cone structure of $T_p^gX$ (see Section \ref{sec:tang}), we have $\lambda B(o,\lambda^{-1}R)=B(o,R)$.
Therefore we obtain the desired convergence $(\lambda X,p)\to (T_p^gX,o)$ (here, Conditions (1) and (2) also imply that every bounded ball in $T_p^gX$ is totally bounded; since $T_p^gX$ is complete by definition, it is proper, and hence the notion of pointed Gromov--Hausdorff convergence is certainly applicable to it).
Note that, if the shortest paths from $p$ are extendable as in Lemma \ref{lem:tang}, then $\log_p$ is surjective and Condition (2) is trivial.

By the definition of the metric $d^*$ of $T_p^gX$, we have
\begin{equation}\label{eq:log1}
d^*(\log_px,\log_py)=\lim_{t\to0}\frac{d(x_t,y_t)}t\le d(x,y),
\end{equation}
where $x_t,y_t$ denote the $t$-intermediate points for $x,y$ with respect to $p$.
Here, by the Busemann convexity, $d(x_t,y_t)/t$ is non-increasing as $t\to 0$ (in particular, the last inequality holds).
This implies that, if Condition (1) holds for $x,y\in B(p,r)$, then it also holds for any $x_t,y_t\in B(p,rt)$, where $0<t<1$.

First we show Condition (1).
Fix $\delta>0$ and take $R>0$ from Condition ($*$).
Let $A\subset B(p,R)$ be a finite $(\delta R)$-dense subset.
 Since $A$ is finite, there exists $t_0>0$ such that for any $z,w\in A$ and $t\le t_0$, Condition (1) holds for any $t$-intermediate points $z_t,w_t\in B(p,tR)$, i.e.,
\begin{equation}\label{eq:log2}
\left|d(z_t,w_t)-d^*(\log_pz_t,\log_pw_t)\right|\le\delta tR.
\end{equation}
Here we used the monotonicity property of Condition (1) explained in the previous paragraph.

\begin{figure}[ht]
\centering
\begin{tikzpicture}
\coordinate[label=below:$p$](p)at(0,0);
\coordinate[label=above:$x$](x)at(120:1);
\coordinate[label=above:$y$](y)at(60:1);
\coordinate[label=left:$u$](u)at(135:2.5);
\coordinate[label=right:$v$](v)at(45:2.5);
\coordinate[label=left:$u_t$](ut)at(135:1);
\coordinate[label=right:$v_t$](vt)at(45:1);
\coordinate[label=left:$z$](z)at(150:2.5);
\coordinate[label=right:$w$](w)at(30:2.5);
\coordinate[label=below:$z_t$](zt)at(150:1);
\coordinate[label=below:$w_t$](wt)at(30:1);

\draw(p)circle[radius=1.5];
\draw(p)circle[radius=3];

\draw(p)to(x);
\draw(p)to(y);
\draw(p)to(u);
\draw(p)to(v);
\draw(p)to(z);
\draw(p)to(w);
\draw[dashed](x)to(ut)to(zt);
\draw[dashed](y)to(vt)to(wt);
\draw[dashed](u)to(z);
\draw[dashed](v)to(w);

\node[below right]at(-45:1.5){$tR$};
\node[below right]at(-45:3){$R$};

\fill(p)circle(1.5pt);
\fill(x)circle(1.5pt);
\fill(y)circle(1.5pt);
\fill(u)circle(1.5pt);
\fill(v)circle(1.5pt);
\fill(ut)circle(1.5pt);
\fill(vt)circle(1.5pt);
\fill(z)circle(1.5pt);
\fill(w)circle(1.5pt);
\fill(zt)circle(1.5pt);
\fill(wt)circle(1.5pt);
\end{tikzpicture}
\caption{}\label{fig:noncol2}
\end{figure}

Suppose $t<t_0$ and let $x,y\in B(p,tR)$ (see Figure \ref{fig:noncol2}).
By Condition ($*$), there exist $u,v\in B(p,R)$ such that $u_t,v_t$ are $(\delta tR)$-close to $x,y$, respectively.
Since $A$ is $(\delta R)$-dense in $B(p,R)$, there exist $z,w\in A$ that are $(\delta R)$-close to $u,v$, respectively.
By the Busemann convexity, $z_t,w_t$ are $(\delta tR)$-close to $u_t,v_t$, respectively.
Therefore, by the triangle inequality, $x,y$ are $(2\delta tR)$-close to $z_t,w_t$, respectively.
In particular,
\begin{equation}\label{eq:log3}
\left|d(x,y)-d(z_t,w_t)\right|\le 4\delta tR.
\end{equation}
Moreover, by the Busemann convexity again, $x_s,y_s$ are $(2\delta tsR)$-close to $(z_t)_s,(w_t)_s$, respectively, where $0<s<1$.
Passing to the limit as in \eqref{eq:log1}, we have
\begin{equation}\label{eq:log4}
\left|d^\ast(\log_px,\log_py)-d^\ast(\log_pz_t,\log_pw_t)\right|\le 4\delta tR.
\end{equation}
Combining \eqref{eq:log2}, \eqref{eq:log3}, and \eqref{eq:log4} gives
\[\left|d(x,y)-d^*(\log_px,\log_py)\right|\le9\delta tR.\]
Thus Condition (1) holds for $r_0:=t_0R$ by replacing $\delta$ with $9\delta$.

Next we show Condition (2).
Fix $\delta>0$ and take $R>0$ from Condition ($*$).
Suppose $r\le R=:r_0$ and let $v\in B(o,r)$.
Since $T_p^gX$ is the completion of the metrization of $\Gamma_p\times[0,\infty)$, in order to prove Condition (2), we may assume that $v=(\gamma,a)$, where $\gamma$ is a shortest path from $p$ and $0<a<r$.
Then there is some $0<t\le1$ for which $\gamma(at)$ is defined.
By Condition ($*$), there exists $x\in B(p,r)$ such that
\[d(x_t,\gamma(at))\le\delta tr.\]
Together with the cone structure of $T_p^gX$ (see Section \ref{sec:tang}) and the inequality \eqref{eq:log1}, this implies that
\[d^*(\log_px,v)=t^{-1}d^*(\log_px_t,\log_p\gamma(at))\le t^{-1}d(x_t,\gamma(at))\le\delta r,\]
as desired.
\end{proof}

By Claims \ref{clm:noncol1} and \ref{clm:noncol2}, the tangent cone of $X$ is unique at almost every point.
This completes the proof of Theorem \ref{thm:noncol}.
\end{proof}

\begin{rem}
The almost extendability of geodesics ($*$) is analogous to the notion of a \textit{$(1,\delta)$-strainer} in Alexandrov geometry, see \cite{BGP92, BBI01}.
See also the recent study of Busemann concave spaces in \cite{Ke19, HY25}.
Any finite-dimensional Alexandrov space satisfies the non-collapsed MCP (\cite[Proposition 2.8]{Oh07}), and the same holds for any Busemann concave space with non-trivial Hausdorff measure (\cite[Proposition 2.23]{Ke19}).
Moreover, in these spaces, the Gromov--Hausdorff tangent cone is unique at every point and has some cone structure (\cite[Theorem 7.8.1]{BGP92}, \cite[Corollary 2.21]{Ke19}).
This enables us to prove a property similar to ($*$).
What we did in Claim \ref{clm:noncol1} can be viewed as a metric measure analog of this argument.
It is also worth mentioning that Condition ($*$) is reminiscent of the ``almost geodesic completeness'' in the context of large-scale geometry, see \cite{On05, GO07, KR21}.
\end{rem}

\begin{rem}\label{rem:mmr}
In the above proof, the result of Magnabosco--Mondino--Rossi \cite{MMR25} (Theorem \ref{thm:mmr}) is essential to find a Banach tangent cone.
Indeed, since the tangent cones are unique almost everywhere as proved above, the result of Le Donne \cite{Le11} (Theorem \ref{thm:le}) shows that almost every tangent cone is a sub-Finsler Carnot group.
Hence, if such a tangent cone inherits the Busemann convexity, one can immediately conclude that it is a Banach space, since any sub-Finsler Carnot group satisfying the Busemann convexity is a Banach space (Lemma \ref{lem:sub}).
However, in the proof of Theorem \ref{thm:noncol}, we do not know if the tangent cone is Busemann.
Therefore, to conclude that the tangent cone is a Banach space, we essentially relied on another fact that any sub-Finsler Carnot group with non-collapsed MCP is a Banach space (\cite[Theorem 1.4]{MMR25}).
\end{rem}

Theorem \ref{thm:mfd} now follows from Theorem \ref{thm:noncol} and Corollary \ref{cor:mfdcol}.

\begin{rem}\label{rem:cat}
In the CAT case, the existence of a manifold point follows from the dimension theory of Kleiner \cite{Kl99}.
See \cite[Theorem 3.15]{KKK22} for more details.
For Busemann spaces, there are currently no such results.
\end{rem}

\section{Flatness without geodesic completeness}\label{sec:bdry}

In this section, we prove Theorem \ref{thm:bdry}, using Theorem \ref{thm:mfd} and modifying the proof of Theorem \ref{thm:main}.

Let $X$ be as in Theorem \ref{thm:bdry}.
By Theorem \ref{thm:mfd}, $X$ is a topological manifold with boundary, whose interior $\In X$ is convex and consists of regular points.
In particular, for any $p\in\In X$, the tangent cone $T_pX$ is a strictly convex Banach space.
We show that the exponential map gives an isometry between a closed convex subset of $T_pX$ and $X$ (see Section \ref{sec:tang} for the exponential map).

The proof is divided into a local part and a global part.
In Section \ref{sec:loc}, we first show that the Busemann concavity (i.e., the opposite inequality to the Busemann convexity) holds locally in the interior of $X$.
In Section \ref{sec:glo}, using the geodesic convexity of the interior (Theorem \ref{thm:inner}), we prove that the Busemann concavity holds globally.
Together with the Busemann convexity, this implies that the exponential map is an isometry.
As in Section \ref{sec:main}, some arguments below are reminiscent of the proofs of Propositions \ref{prop:kkbran} and \ref{prop:kkkconti}.

\subsection{Local argument}\label{sec:loc}

We first establish the local homogeneity of the Hausdorff measure.
Compare with Lemma \ref{lem:hom}.
As in Section \ref{sec:mfd}, we denote by $\In X$ and $\partial X$ the interior and boundary of a manifold $X$, respectively.

\begin{prop}\label{prop:homloc}
Let $X$ be as in Theorem \ref{thm:bdry} and $q\in\In X$.
Suppose $R\le d(q,\partial X)/10$.
Then there exists $C>0$ such that for any $x\in B(q,R)$ and $0<r<R$, we have
\[\Hcal^n(B(x,r))=Cr^n.\]
\end{prop}

Unlike the proof of Lemma \ref{lem:hom}, we cannot extend shortest paths infinitely due to the lack of geodesic completeness.
Instead, we exploit the property of the Hausdorff measure.

\begin{proof}
Since $R\le d(q,\partial X)/10$, local geodesic completeness and the Busemann condition imply that any shortest path starting from a point of $B(q,R)$ is extendable to length $9R$.
In what follows we will frequently use this property.

For any $x,y\in B(q,R)$ and $0\le t\le 1$, we prove the following two equalities:
\begin{gather}
\Hcal^n(B(x,tR))=t^n\Hcal^n(B(x,R)),\label{eq:homloc1}\\
\Hcal^n(B(x,R))=\Hcal^n(B(y,R)).\label{eq:homloc2}
\end{gather}
The desired equality immediately follows from these two equalities by setting $C:=\Hcal^n(B(q,R))/R^n$.

We first show \eqref{eq:homloc1}.
Let $\Phi_t=\Phi_t^x$ be the $t$-contraction map centered at $x$.
By the local geodesic completeness, we have
\begin{equation}\label{eq:homloc1-1}
\Phi_t(B(x,R))=B(x,tR).
\end{equation}
By the Busemann convexity \eqref{eq:buse}, $\Phi_t$ is $t$-Lipschitz, and hence
\begin{equation}\label{eq:homloc1-2}
\Hcal^n(\Phi_t(B(x,R)))\le t^n\Hcal^n(B(x,R)).
\end{equation}
Note that here we used the property of the Hausdorff measure.
On the other hand, the measure contraction property \eqref{eq:mcp} implies
\begin{equation}\label{eq:homloc1-3}
\Hcal^n (\Phi_t(B(x,R)))\ge t^n\Hcal^n(B(x,R)).
\end{equation}
Combining \eqref{eq:homloc1-1}, \eqref{eq:homloc1-2}, and \eqref{eq:homloc1-3} gives the desired equality \eqref{eq:homloc1}.

Next we show \eqref{eq:homloc2}.
By local geodesic completeness, we can find points $x'$ and $y'$ on the extension of the shortest path $xy$ beyond $x$ and $y$, respectively, such that
\[d(x,x')=d(x,y)=d(y,y').\]
See Figure \ref{fig:homloc1}.

\begin{figure}[ht]
\centering
\begin{tikzpicture}
\coordinate[label=below:$x$](x)at(-1.5,0);
\coordinate[label=below:$y$](y)at(1.5,0);
\coordinate[label=below:$x'$](x')at(-4.5,0);
\coordinate[label=below:$y'$](y')at(4.5,0);

\draw(x)circle[radius=1];
\draw(y)circle[radius=1];
\draw(y)circle[radius=0.5];

\draw(x')to(x)to(y)to(y');
\draw[dashed](-1.5,1)to(y');
\draw[dashed](-1.5,-1)to(y');

\fill(x)circle(1.5pt);
\fill(y)circle(1.5pt);
\fill(x')circle(1.5pt);
\fill(y')circle(1.5pt);
\end{tikzpicture}
\caption{}\label{fig:homloc1}
\end{figure}

We show $\Hcal^n(B(x,R))\le\Hcal^n(B(y,R))$.
Let $\Phi_{1/2}=\Phi_{1/2}^{y'}$ be the $(1/2)$-contraction map centered at $y'$.
The measure contraction property \eqref{eq:mcp} implies
\begin{equation}\label{eq:homloc2-1}
\Hcal^n(B(x,R))\le2^n\Hcal^n(\Phi_{1/2}(B(x,R))).
\end{equation}
By the Busemann convexity \eqref{eq:buse}, we have
\begin{equation}\label{eq:homloc2-2}
\Phi_{1/2}(B(x,R))\subset B(y,R/2).
\end{equation}
By the previous equality \eqref{eq:homloc1}, we have
\begin{equation}\label{eq:homloc2-3}
2^n\Hcal ^n(B(y,R/2))=\Hcal^n(B(y,R)).
\end{equation}
Combining \eqref{eq:homloc2-1}, \eqref{eq:homloc2-2}, and \eqref{eq:homloc2-3} shows $\Hcal^n(B(x,R))\le\Hcal^n(B(y,R))$, as desired.
The symmetric procedure using $x'$ gives the opposite inequality, which together implies the desired equality \eqref{eq:homloc2}.
This completes the proof of Proposition \ref{prop:homloc}.
\end{proof}

Using Proposition \ref{prop:homloc}, we show the local Busemann concavity of our Busemann space, i.e., the opposite inequality to the Busemann convexity holds.

\begin{prop}\label{prop:coneloc}
Let $X$ be as in Theorem \ref{thm:bdry} and $p\in\In X$.
Fix $q\in\In X$ and suppose $R\le d(q,\partial X)/10$.
Then for any $x,y\in B(q,R)$ and $0\le t\le 1$ with $x_t,y_t\in B(q,R)$, we have
\[d(x_t,y_t)=td(x,y),\]
where $x_t,y_t$ denote the $t$-intermediate points for $x,y$ with respect to $p$, respectively.
\end{prop}

\begin{proof}
By the Busemann convexity \eqref{eq:buse}, we have $d(x_t,y_t)\le td(x,y)$.
Suppose the opposite inequality does not hold for some $x,y,x_t,y_t\in B(q,R)$, that is,
\[d(x_t,y_t)<td(x,y).\]
Set $D:=d(x,y)/2 <R$.
The above strict inequality shows that $B(x_t,tD)$ and $B(y_t,tD)$ have nonempty intersection (see Figure \ref{fig:coneloc}).
This implies
\begin{equation}\label{eq:coneloc1}
\Hcal^n\left(B(x_t,tD)\cup B(y_t,tD)\right)<2C(tD)^n,
\end{equation}
where $C$ is the constant of Proposition \ref{prop:homloc}.

\begin{figure}[ht]
\centering
\begin{tikzpicture}
\coordinate[label=left:$p$](p)at(0,0);
\coordinate[label=above:$x_t$](xt)at(4,0.5);
\coordinate[label=below:$y_t$](yt)at(4,-0.5);
\coordinate[label=above:$x$](x)at(6,1);
\coordinate[label=below:$y$](y)at(6,-1);

\draw(xt)circle[radius=0.6];
\draw(yt)circle[radius=0.6];
\draw(x)circle[radius=1];
\draw(y)circle[radius=1];
\draw(-20:8)arc(-20:20:8);

\node[above]at(4,1.1){$tD$};
\node[above]at(6,2){$D$};
\node[right]at(8,0){$\partial X$};

\fill(p)circle(1.5pt);
\fill(xt)circle(1.5pt);
\fill(yt)circle(1.5pt);
\fill(x)circle(1.5pt);
\fill(y)circle(1.5pt);

\draw(p)to(xt)to(x);
\draw(p)to(yt)to(y);
\draw[dashed](xt)to(yt);
\draw[dashed](x)to(y);
\end{tikzpicture}
\caption{}\label{fig:coneloc}
\end{figure}

On the other hand, by the triangle inequality, the intersection of $B(x,D)$ and $B(y,D)$ is empty.
By Proposition \ref{prop:homloc}, we have
\begin{equation}\label{eq:coneloc2}
2C(tD)^n=t^n\Hcal^n(B(x,D)\cup B(y,D)).
\end{equation}
By the measure contraction property \eqref{eq:mcp},
\begin{equation}\label{eq:coneloc3}
t^n\Hcal^n(B(x,D)\cup B(y,D))\leq \Hcal^n(\Phi_t(B(x,D)\cup B(y,D))),
\end{equation}
where $\Phi_t$ is the $t$-contraction map centered at $p$.
By the Busemann convexity \eqref{eq:buse},
\begin{equation}\label{eq:coneloc4}
\Phi_t(B(x,D)\cup B(y,D))\subset B(x_t,tD)\cup B(y_t,tD).
\end{equation}
Combining \eqref{eq:coneloc2}, \eqref{eq:coneloc3}, and \eqref{eq:coneloc4} gives a contradiction to \eqref{eq:coneloc1}.
This completes the proof.
\end{proof}

\begin{rem}
If $X$ is geodesically complete, i.e., $\partial X=\emptyset$, Propositions \ref{prop:homloc} and \ref{prop:coneloc} hold globally for $R=\infty$.
In particular, this gives a slightly different proof of Theorem \ref{thm:main} in the non-collapsing case.
\end{rem}

\subsection{Global argument}\label{sec:glo}

Next we establish a Toponogov-type globalization theorem for the Busemann concavity in a slightly general setting.

Recall that for a uniquely geodesic space $Y$, the \textit{Busemann concavity} with respect to $p\in Y$ means the following inequality for $x,y\in Y$ and $0\le t\le 1$:
\begin{equation}\label{eq:topo}
d(x_t,y_t)\ge td(x,y)
\end{equation}
where $x_t,y_t$ are the $t$-intermediate points for $x,y$ with respect to $p$.
\begin{prop}\label{prop:topo}
Let $Y$ be a (not necessarily complete) uniquely geodesic space such that the shortest paths depend continuously on their endpoints (i.e., the mapping $(x,y,t)\mapsto \gamma_{xy}(t)$ is continuous in $x,y\in Y$ and $t\in[0,1]$, where $\gamma_{xy}$ is the linear reparameterization of the shortest path $xy$ defined on $[0,1]$).
Fix $p\in Y$ and suppose that any $q\in Y$ has a neighborhood $U$ satisfying the following:
\begin{enumerate}
\item for any $x\in U$, the shortest path $px$ extends to length $d(p,x)+\epsilon$ beyond $x$, where $\epsilon>0$ is independent of $x$;
\item for any $x,y\in U$ and any $0\le t\le 1$ with $x_t,y_t\in U$, the Busemann concavity \eqref{eq:topo} holds.
\end{enumerate}
Then the Busemann concavity \eqref{eq:topo} holds for any $x,y\in Y$ and $0\le t\le 1$.
\end{prop}

\noindent Later we will regard $Y$ as the interior of $X$ of Theorem \ref{thm:bdry}.
This is why we did not assume the completeness of $Y$.

The following proof is reminiscent of the final step of the proof of the Cartan--Hadamard-type globalization theorem for the Busemann convexity \cite{AB90} (see, e.g., \cite[Lemma 7.14]{FG25}).

\begin{proof}
Let $x,y\in Y$ be arbitrary.
We first prove \eqref{eq:topo} in the special case where $t$ is sufficiently close to $1$.
Let $\gamma$ (resp.\ $\gamma_t$) be the shortest path between $x$ and $y$ (resp.\ $x_t$ and $y_t$).

By the compactness of $\gamma$, we take a finite cover $\{U_\alpha\}$ of $\gamma$ such that each $U_\alpha$ satisfies the assumptions (1) and (2).
Let $\epsilon_\alpha>0$ be the extendable length for each $U_\alpha$ as in (1).
By the finiteness of $\{U_\alpha\}$ and the compactness of $\gamma$, we choose $0<\epsilon\le\min\{\epsilon_\alpha\}$ such that for any $w\in\gamma$, there exists $\alpha$ such that $B(w,3\epsilon)\subset U_\alpha$.
We then choose $t$ sufficiently close to $1$ so that $\gamma_t$ is $\epsilon$-close to $\gamma$ and is contained in the union of $U_\alpha$, and $(t^{-1}-1)d(p,z)<\epsilon$ for all $z\in \gamma_t$.
This choice is possible by the continuity of shortest paths as in the assumption.

\begin{figure}[ht]
\centering
\begin{tikzpicture}
\coordinate[label=left:$p$](p)at(0,0);
\coordinate[label=above:$x_t$](z0)at(6,2);
\coordinate(z1)at(6,1);
\coordinate[label=below left:$z_i$](z2)at(6,0);
\coordinate(z3)at(6,-1);
\coordinate[label=below:$y_t$](z4)at(6,-2);
\coordinate[label=above:$x$](w0)at($(p)!1.1!(z0)$);
\coordinate(w1)at($(p)!1.2!(z1)$);
\coordinate[label=below right:$w_i$](w2)at($(p)!1.2!(z2)$);
\coordinate(w3)at($(p)!1.2!(z3)$);
\coordinate[label=below:$y$](w4)at($(p)!1.1!(z4)$);

\draw(z0)to(z4);
\draw(w0)to(w4);
\draw[dashed](p)to(w0);
\draw[dashed](p)to(w1);
\draw[dashed](p)to(w2);
\draw[dashed](p)to(w3);
\draw[dashed](p)to(w4);
\draw(6.5,1.6)circle[radius=1.1];

\node[right]at(7.7,1.6){$U_\alpha$};
\node[left]at(6,-1.5){$\gamma_t$};
\node[right]at(6.6,-1.7){$\gamma$};

\fill(p)circle(1.5pt);
\fill(z0)circle(1.5pt);
\fill(z1)circle(1.5pt);
\fill(z2)circle(1.5pt);
\fill(z3)circle(1.5pt);
\fill(z4)circle(1.5pt);
\fill(w0)circle(1.5pt);
\fill(w1)circle(1.5pt);
\fill(w2)circle(1.5pt);
\fill(w3)circle(1.5pt);
\fill(w4)circle(1.5pt);
\end{tikzpicture}
\caption{}\label{fig:topo}
\end{figure}

Subdivide $\gamma_t$ by finitely many points $x_t=z_0, z_1,\dots, z_{N-1}, z_N=y_t$ such that $d(z_i,z_{i+1})<\epsilon$ (see Figure \ref{fig:topo}).
By the choices of $t$ and $\epsilon$, we can extend the shortest path $pz_i$ beyond $z_i$ to a shortest path $pw_i$ of length $t^{-1}d(p,z_i)$ (in particular, we choose $w_0=x$ and $w_N=y$).
Observe that the four points $z_i,z_{i+1},w_i,w_{i+1}$ lie in a common ball of radius $3\epsilon$ centered at a point of $\gamma$.
Hence they are contained in a single $U_\alpha$.
Then the assumption (2) implies the Busemann concavity \eqref{eq:topo} for $w_i,w_{i+1}$ and $t$, that is,
\[d(z_i,z_{i+1})\ge td(w_i,w_{i+1}).\]
Adding these inequalities for all $i$ gives us
\begin{align*}
d(x_t,y_t)&=d(z_0,z_1)+\dots+d(z_{N-1},z_N)\\
&\ge t(d(w_0,w_1)+\dots+d(w_{N-1},w_N))\\
&\ge td(x,y),
\end{align*}
where we used the triangle inequality in the last step.

Now we prove the general case.
For any fixed $x,y\in Y$, let $\tau$ be the infimum of $0\le t_0\le 1$ such that the Busemann concavity \eqref{eq:topo} holds for every $t\in[t_0,1]$.
By continuity,
$\tau$ is actually the minimum, and thus
\begin{equation}\label{eq:topo1}
d(x_\tau,y_\tau)\ge\tau d(x,y).
\end{equation}
Suppose $\tau>0$.
By the special case in the previous paragraph, for any $0<s<1$ sufficiently close to $1$, we have
\begin{equation}\label{eq:topo2}
d((x_\tau)_s,(y_\tau)_s)\ge sd(x_\tau,y_\tau).
\end{equation}
Since $(x_\tau)_s=x_{\tau s}$ and $(y_\tau)_s=y_{\tau s}$, combining \eqref{eq:topo1} and \eqref{eq:topo2} gives
\[d(x_{\tau s},y_{\tau s})\ge(\tau s)d(x,y),\]
which contradicts the minimality of $\tau$.
This completes the proof.
\end{proof}

Now we are ready to prove Theorem \ref{thm:bdry}.

\begin{proof}[Proof of Theorem \ref{thm:bdry}]
Let $X$ be as in Theorem \ref{thm:bdry} and $p\in\In X$.
By Theorem \ref{thm:mfd}, $p$ is an $n$-regular point, i.e., the tangent cone $T_pX$ is a strictly convex Banach space of dimension $n$.
Let $\exp_p:A\to X$ be the exponential map (see Section \ref{sec:tang}), where $A$ is the maximal subset of $T_pX$ for which $\exp_p$ is defined.

We prove that the Busemann concavity \eqref{eq:topo} with respect to $p$ holds for any $x,y\in X$ and $0\le t\le 1$.
Then $\exp_p$ is an isometry from $A$ to $X$.
Since $T_pX$ is a strictly convex Banach space (hence uniquely geodesic) and $X$ is a complete geodesic space, we see that $A=\exp_p^{-1}(X)$ is a closed convex subset of $T_pX$.

We will apply Proposition \ref{prop:topo} to $\In X$ for proving the Busemann concavity in $\In X$.
Since $\In X$ is convex in $X$ (Theorem \ref{thm:mfd}), it is an (incomplete) uniquely geodesic space.
Furthermore, shortest paths vary continuously in $\In X$ (by the Arzel\`a--Ascoli theorem or the convexity of distance function, Remark \ref{rem:buse}).
Hence it remains to show that for any $q\in\In X$, there exists a neighborhood $U$ satisfying the assumptions of Proposition \ref{prop:topo}.
The assumption (1) follows from Proposition \ref{prop:ls} and the fact that any locally geodesically complete, locally Busemann space has local uniform extendability of shortest paths; see Section \ref{sec:buse}.
The assumption (2) was already proved in Proposition \ref{prop:coneloc}.
Therefore, by Proposition \ref{prop:topo}, the Busemann concavity with respect to $p$ holds in $\In X$.

Finally, since $\In X$ is dense in $X$ and shortest paths vary continuously in $X$, we get the Busemann concavity over $X$.
This completes the proof.
\end{proof}

\begin{rem}
A characterization of a convex subset in a strictly convex Banach space in terms of affine functions can be found in \cite{HiL07}.
The authors are not aware of an alternative proof of Theorem \ref{thm:bdry} using this result.
\end{rem}

Finally, we give a simple example showing that the non-collapsing assumption of Theorem \ref{thm:bdry} is necessary, as noted in Remark \ref{rem:bdry} (cf.\ \cite[Remark 5.6]{St06}).

\begin{ex}\label{ex:hyp}

Let $\mathbb H^2$ be the hyperbolic plane and $\bar B(R)$ its closed $R$-ball.
Since $\bar B(R)$ is convex in $\mathbb H^2$, it is CAT($0$).
Furthermore, for any $N>2$, there exists $R>0$ such that $\bar B(R)$ with the Hausdorff measure $\Hcal^2$ satisfies MCP($0,N$).
Indeed, since $\bar B(R)$ is convex in $\mathbb H^2$, it satisfies MCP($-1,2$).
That is, for any $p\in\bar B(R)$ and any measurable $A\subset\bar B(R)$, we have
\[\Hcal^2(A_t)\ge\int_{A}t\frac{\sinh(td(p,x))}{\sinh(d(p,x))}d\Hcal^2(x)\]
for all $0\le t\le 1$ (see Remark \ref{rem:mcp}).
Here, if $R$ is sufficiently small compared to $N>2$, the integrand is not less than $t^N$.
We will check it for $N\in(2,3)$ (the $N\ge 3$ case follows by monotonicity).
Setting $r:=d(p,x)$ and $f_r(t):=\sinh(tr)-t^{N-1}\sinh(r)$, it suffices to prove that $f_r(t)\ge 0$ for all $t\in[0,1]$, provided $r$ is sufficiently small.
A direct computation shows
\[f_r''(t)=r^2\sinh(tr)-(N-1)(N-2)t^{N-3}\sinh(r),\]
which is non-positive for $t\in[0,1]$, provided $N\in(2,3)$ and $r$ is sufficiently small.
Since $f_r(0)=f_r(1)=0$, we get $f_r(t)\ge0$ for all $t\in[0,1]$, as desired.
Therefore, $\Hcal^2(A_t)\ge t^N\Hcal^2(A)$, i.e., $\bar B(R)$ satisfies MCP($0,N$).
\end{ex}

\begin{rem}\label{rem:grushin}

Even if the MCP in Theorem \ref{thm:bdry} is replaced by the CD condition, the non-collapsing assumption is necessary.
For example, the so-called Grushin halfplane with a suitable measure satisfies both CAT($0$) and CD($0,N$) (and hence the same is true for any proper convex subset of the Grushin halfplane).
See \cite[Remark 3.10]{DHPW23} and references therein.
Note that, as stated there, the Grushin halfplane is also an example for which the topological dimension is not equal to the Hausdorff dimension in the CAT+CD setting.
\end{rem}

\section{Almost flat structure}\label{sec:int}

Finally, we prove Theorem \ref{thm:int}, modifying the proof of the local part of Theorem \ref{thm:bdry} given in Section \ref{sec:loc}.

Let $0<\epsilon<1$.
We say that a map $f:X\to Y$ between metric spaces $(X,d_X)$ and $(Y,d_Y)$ is an \textit{$\epsilon$-almost isometry} if it is surjective and
\[\left|\frac{d_Y(f(x),f(x'))}{d_X(x,x')}-1\right|\le\epsilon\]
for any $x\neq x'\in X$.

The following theorem is the main result of this section.

\begin{thm}\label{thm:exp}
Let $X$ be a locally Busemann space satisfying local non-collapsed {\rm MCP($K,n$)}, where $K\le 0$ and $n\ge 1$.
Let $p\in X$ be a manifold point.
Then, for any $\epsilon>0$, there exists $R>0$ such that the exponential map 
\[\exp_p:B(o,R)\to  B(p,R)\]
is an $\epsilon$-almost isometry, where $o$ denotes the apex of the tangent cone $T_pX$.
\end{thm}

\noindent Note that $T_pX=T_p^gX$ is isometric to a strictly convex Banach space of dimension $n$, by Corollary \ref{cor:reg}.
In particular, Theorem \ref{thm:exp} implies Theorem \ref{thm:int}.

The proof of Theorem \ref{thm:exp} is a minor modification of that of Proposition \ref{prop:coneloc}, where we proved that $\exp_p:B(o,R)\to B(p,R)$ is an isometry under the assumption $K=0$.
Indeed, Theorem \ref{thm:exp} can be viewed as the local almost rigid version of Theorem \ref{thm:bdry}.

For the proof, we first introduce the ``almost MCP($0,N$) condition'', which was already used in the proof of Claim \ref{clm:noncol1}.
Let $(X,d,m)$ be an MCP($K,N$) space, where $K<0$.
For given small $\delta>0$, assume that the diameter of $X$ is sufficiently small depending on $\delta$ (and $K,N$).
Then, for any $p\in X$ and any measurable $A\subset X$, we have
\begin{equation}\label{eq:mcp'}
m(A_t)\ge(1-\delta) t^Nm(A)
\end{equation}
for all $0\le t\le 1$, where $A_t$ denotes the $t$-intermediate set for $A$ with respect to $p$.
This immediately follows from the definition of the MCP($K,N$) in Remark \ref{rem:mcp} and the assumption that $\diam X\ll\delta$.
We refer to \eqref{eq:mcp'} as the \textit{$\delta$-almost {\rm MCP($0,N$)} condition}.

\begin{proof}[Proof of Theorem \ref{thm:exp}]
Let $X$ and $p$ be as in Theorem \ref{thm:exp}.
Fix $\epsilon>0$ and suppose $R>0$ is small enough, depending on $\epsilon$ and $p$ (in particular, we assume that $B(p,10R)$ satisfies the Busemann convexity and the measure contraction property).
We prove that for any $x,y\in B(p,R)$ and $0\le t\le 1$, the following ``almost Busemann concavity'' holds:
\begin{equation}\label{eq:exp}
d(x_t,y_t)\ge(1-\epsilon)td(x,y),
\end{equation}
where $x_t,y_t$ denote the $t$-intermediate points for $x,y$ with respect to $p$, respectively.

Indeed, once the inequality \eqref{eq:exp} is proved, dividing both sides by $t$ and taking $t\to0$, we obtain
\[d^*(\log_px,\log_py)\ge(1-\epsilon)d(x,y),\]
where $d^*$ denotes the metric of the tangent cone and $\log_p$ denotes the logarithmic map (see the proof of Claim \ref{clm:noncol2}).
Furthermore, by the Busemann convexity, we have $d^*(\log_px,\log_py)\le d(x,y)$.
Therefore $\log_p$ is an $\epsilon$-almost isometry (the surjectivity of $\log_p$ clearly follows from local geodesic completeness).
Its inverse $\exp_p$ is an $\epsilon/(1-\epsilon)$-almost isometry, which proves the claim after replacing $\epsilon$.

Let us show \eqref{eq:exp}.
Suppose $\delta>0$ is sufficiently small compared to $\epsilon$ (which will be determined at the end of the proof).
Choosing $R$ small enough, we may assume that $B(p,10R)$ satisfies the $\delta$-almost non-collapsed MCP($0,n$), as in \eqref{eq:mcp'}.

We first show the local ``almost homogeneity'' of the Hausdorff measure.
Compare with Proposition \ref{prop:homloc}.

\begin{clm}\label{clm:homloc'}
Suppose $R\le d(p,\partial X)/10$ and $B(p,10R)$ satisfies the Busemann convexity and the $\delta$-almost non-collapsed {\rm MCP($0,n$)}.
Then there exists $C>0$ such that for any $x\in B(p,R)$ and $0<r<R$, we have
\[(1-\kappa(\delta))Cr^n\le\Hcal^n(B(x,r))\le(1+\kappa(\delta))Cr^n,\]
where $\kappa(\delta)$ is a positive function depending only on $\delta$ such that $\kappa(\delta)\to0$ as $\delta\to0$.
\end{clm}

\begin{proof}
The proof is almost identical to Proposition \ref{prop:homloc}, except that the MCP($0,n$) condition is replaced by the almost MCP($0,n$) condition.

Indeed, \eqref{eq:homloc1-1} and \eqref{eq:homloc1-2} are true without any changes.
The inequality \eqref{eq:homloc1-3} holds with a small error:
\[\Hcal^n(\Phi_t(B(x,R)))\ge(1-\delta)t^n\Hcal^n(B(x,R)).\]
This changes the equality \eqref{eq:homloc1} to the following almost equality:
\begin{equation}\label{eq:homloc1'}
(1-\delta)t^n\Hcal^n(B(x,R))\le\Hcal^n(B(x,tR))\le t^n\Hcal^n(B(x,R)).
\end{equation}
Similarly, the inequality \eqref{eq:homloc2-1} is modified as follows:
\[(1-\delta)\Hcal^n(B(x,R))\le2^n\Hcal^n(\Phi_{1/2}(B(x,R))).\]
The inequality \eqref{eq:homloc2-2} remains true.
Furthermore, the necessary side of the equality \eqref{eq:homloc2-3} also remains true by the right inequality of \eqref{eq:homloc1'}:
\[2^n\Hcal ^n(B(y,R/2))\le\Hcal^n(B(y,R)).\]
These modifications change the equality \eqref{eq:homloc2} to:
\begin{equation}\label{eq:homloc2'}
(1-\delta)\Hcal^n(B(x,R))\le\Hcal^n(B(y,R))\le(1-\delta)^{-1}\Hcal^n(B(x,R)).
\end{equation}
Combining \eqref{eq:homloc1'} and \eqref{eq:homloc2'} gives the desired inequality with $C:=\Hcal^n(B(p,R))/R^n$ and a suitable $\kappa(\delta)$.
\end{proof}

Now suppose \eqref{eq:exp} does not hold for some $x,y\in B(p,R)$ and $0<t<1$, i.e.,
\begin{equation}\label{eq:exp'}
d(x_t,y_t)<(1-\epsilon)td(x,y).
\end{equation}
In what follows, we repeat the same argument as in the proof of Proposition \ref{prop:coneloc} by using Claim \ref{clm:homloc'} instead of Proposition \ref{prop:homloc}.

Set $D:=d(x,y)/2$.
The inequality \eqref{eq:exp'} shows that the intersection of $B(x_t,tD)$ and $B(y_t,tD)$ contains the $(\epsilon tD)$-ball centered at the midpoint of the shortest path $x_ty_t$.
Therefore, instead of \eqref{eq:coneloc1}, we have
\begin{equation}\label{eq:coneloc1'}
\Hcal^n\left(B(x_t,tD)\cup B(y_t,tD)\right)\le(1+\kappa(\delta))2C(tD)^n-(1-\kappa(\delta))C(\epsilon tD)^n,
\end{equation}
where $C$ is the constant of Claim \ref{clm:homloc'}.

The other estimates also hold with small errors.
By the triangle inequality, the intersection of $B(x,D)$ and $B(y,D)$ is empty.
By Claim \ref{clm:homloc'} again, we have
\begin{equation}\label{eq:coneloc2'}
(1-\kappa(\delta))2C(tD)^n\le t^n\Hcal^n(B(x,D)\cup B(y,D)).
\end{equation}
By the almost MCP($0,n$) condition \eqref{eq:mcp'},
\begin{equation}\label{eq:coneloc3'}
(1-\delta)t^n\Hcal^n(B(x,D)\cup B(y,D))\leq\Hcal^n(\Phi_t(B(x,D)\cup B(y,D))),
\end{equation}
where $\Phi_t$ is the $t$-contraction map centered at $p$.
By the Busemann convexity \eqref{eq:buse},
\begin{equation}\label{eq:coneloc4'}
\Phi_t(B(x,D)\cup B(y,D))\subset B(x_t,tD)\cup B(y_t,tD).
\end{equation}
Combining \eqref{eq:coneloc2'}, \eqref{eq:coneloc3'}, and \eqref{eq:coneloc4'} gives
\[(1-\delta)(1-\kappa(\delta))2C(tD)^n\le\Hcal^n(B(x_t,tD)\cup B(y_t,tD)).\]
Taking $R$ sufficiently small so that $\delta,\kappa(\delta)\ll\epsilon$, we get a contradiction to \eqref{eq:coneloc1'} (recall that $\kappa(\delta)\to0$ as $\delta\to0$ as $R\to0$).
This completes the proof.
\end{proof}

\begin{rem}
As can be seen from the above proof, the choice of $R$ depends only on the distance to the boundary of $X$ and the radius for which the Busemann and non-collapsed MCP($K,n$) conditions hold.
In particular, Theorem \ref{thm:exp} can be strengthened to the following local uniform version:

Let $X$ be as in Theorem \ref{thm:exp}.
Then for any $p\in\In X$ and any $\epsilon>0$, there exists $R>0$ such that $\exp_q:B(o_q,R)\to B(q,R)$ is an $\epsilon$-almost isometry for any $q\in B(p,R)$, where $o_q$ denotes the apex of $T_qX$.
\end{rem}

\section{Problems}\label{sec:prob}

In this section we summarize the remaining open problems, some of which have already been discussed.

The first problem is to drop the non-collapsing assumption of Theorem \ref{thm:mfd}, as mentioned in Remark \ref{rem:mfd}.
Note that this is impossible for Theorem \ref{thm:bdry}, see Remark \ref{rem:bdry}.
In view of Corollary \ref{cor:mfdcol}, it suffices to consider the following.

\begin{prob}\label{prob:mfd}
Let $X$ be a locally Busemann space satisfying local (collapsed) MCP($K,N$), where $K\le 0$ and $N\ge 1$.
Show that $X$ contains a manifold point of some dimension.
\end{prob}

\noindent One direction towards this problem is to develop a Busemann analog of Kleiner's dimension theory for CAT spaces \cite{Kl99}, which had already been partly discussed in \cite[Theorem D]{Kl99}.

The second problem is a collapsing version of Theorem \ref{thm:int}.
Since the interior of the manifold $X$ of Corollary \ref{cor:mfdcol} is locally geodesically complete, we will assume local geodesic completeness from the beginning (and instead, we may drop completeness).

\begin{ques}
Let $X$ be a locally geodesically complete, locally Busemann space with local (collapsed) MCP($K,N$).
Does every point of $X$ have a neighborhood almost isometric to an open subset of a strictly convex Banach space (via the exponential map)?
\end{ques}

\noindent In particular, the authors do not know if it is possible to develop the almost rigid version of Theorem \ref{thm:main}.
Partial results would follow from the more general theory of geodesically complete Busemann spaces.

The third problem is a generalization of the main result of \cite{KK20}, which is also an improvement of Theorem \ref{thm:exp}.
In \cite{KK20}, Kapovitch--Ketterer proved that if a non-collapsed CD($K,n$) space satisfies the CAT($\kappa$) condition locally, then it is an Alexandrov space with curvature bounded below by $K-\kappa(n-2)$.
In other words, the synthetic sectional and Ricci curvatures satisfy the same relation as in the Riemannian case.

Recall that, in the proof of Theorem \ref{thm:bdry}, we proved that the Busemann convexity combined with the non-collapsed MCP($0,n$) implies the Busemann concavity.
This can be viewed as the generalization of the Kapovitch--Ketterer theorem in the extremal case $K=\kappa=0$.

In order to discuss further generalization, one needs to define general lower curvature bounds in the sense of Busemann, which generalizes Alexandrov lower curvature bounds.
This is done by regarding the Busemann concavity \eqref{eq:conc} as a comparison inequality with Euclidean model space and replacing the model space with the hyperbolic one of constant curvature $K<0$.
We call it the \textit{$K$-Busemann concavity}.
For more details, see \cite[Section 2]{Ke19}.

Now we can ask the following question.

\begin{ques}
Let $X$ be a locally Busemann space satisfying local non-collapsed MCP($K,n$).
Does $X$ satisfy the $K$-Busemann concavity locally?
\end{ques}

\noindent Similarly, one can define the \textit{$\kappa$-Busemann convexity} (see \cite[Section 2.4]{FG25}), which generalizes the CAT($\kappa$) condition.
It is also natural to ask whether our main results extend to $\kappa$-Busemann convex spaces, but we will not pursue this further here.

Another problem is the construction of a Finsler metric in the manifold part of a Busemann space with MCP, analogous to the Riemannian structure in the CAT with CD setting (\cite[Corollary 1.2]{KK20}, \cite[Theorem 1.1]{KKK22}).
As before, since we are only interested in the interior, we assume local geodesic completeness from the beginning.

\begin{prob}\label{prob:fin}
Let $X$ be a locally geodesically complete, locally Busemann space with local MCP($K,N$) (non-collapsed if necessary).
Construct a Finsler metric on $X$ that is compatible with the original distance.
\end{prob}

\noindent
However, at present, we suspect that this requires some more regularity for the distance function, which is independent of MCP (cf.\ \cite{Po90, Po98}).
The details will be discussed in future work.

Finally, we ask about the Berwaldness of a Busemann space with MCP.
Recall that Ivanov--Lytchak \cite{IL19} proved that any locally Busemann Finsler manifold is \textit{Berwald}, that is, its geodesics are affine reparameterizations of geodesics of some Riemannian metric.
In particular, this implies that all tangent norms are mutually isometric.
Compare with Proposition \ref{prop:kkkconti}.

\begin{ques}\label{ques:ber}
Let $X$ be a locally geodesically complete, locally Busemann space with local MCP($K,N$) (non-collapsed if necessary).
Does $X$ satisfy any kind of Berwaldness?
For example, are the tangent cones mutually isometric?
\end{ques}

Note that all the above problems and questions remain open even if the MCP in the assumption is replaced by the CD condition.

\printbibliography

@article {AB90,
    AUTHOR = {Alexander, Stephanie B. and Bishop, Richard L.},
     TITLE = {The {H}adamard-{C}artan theorem in locally convex metric
              spaces},
   JOURNAL = {Enseign. Math. (2)},
  FJOURNAL = {L'Enseignement Math\'ematique. Revue Internationale. 2e
              S\'erie},
    VOLUME = {36},
      YEAR = {1990},
    NUMBER = {3-4},
     PAGES = {309--320},
      ISSN = {0013-8584},
   MRCLASS = {53C70 (53C22)},
  MRNUMBER = {1096422},
MRREVIEWER = {Werner\ Ballmann},
}

@book {AKP24,
    AUTHOR = {Alexander, Stephanie and Kapovitch, Vitali and Petrunin,
              Anton},
     TITLE = {Alexandrov geometry---foundations},
    SERIES = {Graduate Studies in Mathematics},
    VOLUME = {236},
 PUBLISHER = {American Mathematical Society, Providence, RI},
      YEAR = {2024},
     PAGES = {xvii+282},
   MRCLASS = {53C23 (30Lxx 53-02 53C45)},
  MRNUMBER = {4734965},
}

@article {AmbGigMonRaj15,
    AUTHOR = {Ambrosio, Luigi and Gigli, Nicola and Mondino, Andrea and
              Rajala, Tapio},
     TITLE = {Riemannian {R}icci curvature lower bounds in metric measure
              spaces with {$\sigma$}-finite measure},
   JOURNAL = {Trans. Amer. Math. Soc.},
  FJOURNAL = {Transactions of the American Mathematical Society},
    VOLUME = {367},
      YEAR = {2015},
    NUMBER = {7},
     PAGES = {4661--4701},
      ISSN = {0002-9947,1088-6850},
   MRCLASS = {49J52 (31C25 35K90 49Q20 58J35)},
  MRNUMBER = {3335397},
MRREVIEWER = {Paul\ Bryan},
       DOI = {10.1090/S0002-9947-2015-06111-X},
       URL = {https://doi.org/10.1090/S0002-9947-2015-06111-X},
}

@article {AmbGigSav14,
    AUTHOR = {Ambrosio, Luigi and Gigli, Nicola and Savar\'e, Giuseppe},
     TITLE = {Metric measure spaces with {R}iemannian {R}icci curvature
              bounded from below},
   JOURNAL = {Duke Math. J.},
  FJOURNAL = {Duke Mathematical Journal},
    VOLUME = {163},
      YEAR = {2014},
    NUMBER = {7},
     PAGES = {1405--1490},
      ISSN = {0012-7094,1547-7398},
   MRCLASS = {35R01 (60J45 60J65)},
  MRNUMBER = {3205729},
       DOI = {10.1215/00127094-2681605},
       URL = {https://doi.org/10.1215/00127094-2681605},
}

@article {An09,
    AUTHOR = {Andreev, P. D.},
     TITLE = {Geometric constructions in the class of {B}usemann
              nonpositively curved spaces},
   JOURNAL = {J. Math. Phys. Anal. Geom.},
  FJOURNAL = {Journal of Mathematical Physics, Analysis, Geometry},
    VOLUME = {5},
      YEAR = {2009},
    NUMBER = {1},
     PAGES = {25--37, 107},
      ISSN = {1812-9471,1817-5805},
   MRCLASS = {53C70 (53C23)},
  MRNUMBER = {2528398},
MRREVIEWER = {Joseph\ E.\ Borzellino},
}

@article {An17,
    AUTHOR = {Andreev, P. D.},
     TITLE = {Normed space structure on a {B}usemann {$G$}-space of cone
              type},
   JOURNAL = {Mat. Zametki},
  FJOURNAL = {Matematicheskie Zametki},
    VOLUME = {101},
      YEAR = {2017},
    NUMBER = {2},
     PAGES = {169--180},
      ISSN = {0025-567X,2305-2880},
   MRCLASS = {53C70},
  MRNUMBER = {3608015},
       DOI = {10.4213/mzm10609},
       URL = {https://doi.org/10.4213/mzm10609},
addendum = {Translation in \textit{Math. Notes} \textbf{101}.1-2 (2017), 193--202}
}

@article {AS19,
    AUTHOR = {Andreev, P. D. and Starostina, V. V.},
     TITLE = {Normed planes in the tangent cone to a chord space of
              nonpositive curvature},
   JOURNAL = {Izv. Vyssh. Uchebn. Zaved. Mat.},
  FJOURNAL = {Izvestiya Vysshikh Uchebnykh Zavedeni\u i. Matematika.
              Kazanski\u i\ Gosudarstvenny\u i\ Universitet},
      YEAR = {2019},
    NUMBER = {1},
     PAGES = {3--17},
      ISSN = {0021-3446,2076-4626},
   MRCLASS = {53C23},
  MRNUMBER = {3971763},
addendum = {Translation in \textit{Russian Math. (Iz. VUZ)} \textbf{63}.1 (2019), 1--13}
}

@article {BarRiz18,
    AUTHOR = {Barilari, Davide and Rizzi, Luca},
     TITLE = {Sharp measure contraction property for generalized {H}-type
              {C}arnot groups},
   JOURNAL = {Commun. Contemp. Math.},
  FJOURNAL = {Communications in Contemporary Mathematics},
    VOLUME = {20},
      YEAR = {2018},
    NUMBER = {6},
     PAGES = {1750081, 24},
      ISSN = {0219-1997,1793-6683},
   MRCLASS = {53C17 (35R03 53C21 53C22 54E35)},
  MRNUMBER = {3848070},
MRREVIEWER = {Nicolas\ Juillet},
       DOI = {10.1142/S021919971750081X},
       URL = {https://doi.org/10.1142/S021919971750081X},
}

@misc{BMRT24a,
      title={Measure contraction property and curvature-dimension condition on sub-{F}insler {H}eisenberg groups}, 
      author={Borza, S. and Magnabosco, M. and Rossi, T. and Tashiro, T.},
      year={2024},
      eprint={2402.14779v1},
      archivePrefix={arXiv},
      primaryClass={math.MG},
      note={Preprint}
}

@article {Bo95,
    AUTHOR = {Bowditch, B. H.},
     TITLE = {Minkowskian subspaces of non-positively curved metric spaces},
   JOURNAL = {Bull. London Math. Soc.},
  FJOURNAL = {The Bulletin of the London Mathematical Society},
    VOLUME = {27},
      YEAR = {1995},
    NUMBER = {6},
     PAGES = {575--584},
      ISSN = {0024-6093,1469-2120},
   MRCLASS = {20F32 (53C70)},
  MRNUMBER = {1348712},
MRREVIEWER = {Michael\ L.\ Mihalik},
       DOI = {10.1112/blms/27.6.575},
       URL = {https://doi.org/10.1112/blms/27.6.575},
}

@book {BH99,
    AUTHOR = {Bridson, Martin R. and Haefliger, Andr\'e},
     TITLE = {Metric spaces of non-positive curvature},
    SERIES = {Grundlehren der mathematischen Wissenschaften [Fundamental
              Principles of Mathematical Sciences]},
    VOLUME = {319},
 PUBLISHER = {Springer-Verlag, Berlin},
      YEAR = {1999},
     PAGES = {xxii+643},
      ISBN = {3-540-64324-9},
   MRCLASS = {53C23 (20F65 53C70 57M07)},
  MRNUMBER = {1744486},
MRREVIEWER = {Athanase\ Papadopoulos},
       DOI = {10.1007/978-3-662-12494-9},
       URL = {https://doi.org/10.1007/978-3-662-12494-9},
}

@book {BBI01,
    AUTHOR = {Burago, Dmitri and Burago, Yuri and Ivanov, Sergei},
     TITLE = {A course in metric geometry},
    SERIES = {Graduate Studies in Mathematics},
    VOLUME = {33},
 PUBLISHER = {American Mathematical Society, Providence, RI},
      YEAR = {2001},
     PAGES = {xiv+415},
      ISBN = {0-8218-2129-6},
   MRCLASS = {53C23},
  MRNUMBER = {1835418},
MRREVIEWER = {Mario\ Bonk},
       DOI = {10.1090/gsm/033},
       URL = {https://doi.org/10.1090/gsm/033},
}

@article {BS20,
    AUTHOR = {Bru\'e, Elia and Semola, Daniele},
     TITLE = {Constancy of the dimension for {${\rm RCD}(K,N)$} spaces via
              regularity of {L}agrangian flows},
   JOURNAL = {Comm. Pure Appl. Math.},
  FJOURNAL = {Communications on Pure and Applied Mathematics},
    VOLUME = {73},
      YEAR = {2020},
    NUMBER = {6},
     PAGES = {1141--1204},
      ISSN = {0010-3640,1097-0312},
   MRCLASS = {53C23 (53E99)},
  MRNUMBER = {4156601},
MRREVIEWER = {Luis\ Guijarro},
       DOI = {10.1002/cpa.21849},
       URL = {https://doi.org/10.1002/cpa.21849},
}

@article {Bu48,
    AUTHOR = {Busemann, Herbert},
     TITLE = {Spaces with non-positive curvature},
   JOURNAL = {Acta Math.},
  FJOURNAL = {Acta Mathematica},
    VOLUME = {80},
      YEAR = {1948},
     PAGES = {259--310},
      ISSN = {0001-5962,1871-2509},
   MRCLASS = {53.0X},
  MRNUMBER = {29531},
MRREVIEWER = {J.\ J.\ Stoker},
       DOI = {10.1007/BF02393651},
       URL = {https://doi.org/10.1007/BF02393651},
}

@book {Bu55,
    AUTHOR = {Busemann, Herbert},
     TITLE = {The geometry of geodesics},
 PUBLISHER = {Academic Press, Inc., New York},
      YEAR = {1955},
     PAGES = {x+422},
   MRCLASS = {53.0X},
  MRNUMBER = {75623},
MRREVIEWER = {L.\ W.\ Green},
}

@article {CM21,
    AUTHOR = {Cavalletti, Fabio and Milman, Emanuel},
     TITLE = {The globalization theorem for the curvature-dimension
              condition},
   JOURNAL = {Invent. Math.},
  FJOURNAL = {Inventiones Mathematicae},
    VOLUME = {226},
      YEAR = {2021},
    NUMBER = {1},
     PAGES = {1--137},
      ISSN = {0020-9910,1432-1297},
   MRCLASS = {49Q22 (49Q20 53C23)},
  MRNUMBER = {4309491},
MRREVIEWER = {Luca\ Granieri},
       DOI = {10.1007/s00222-021-01040-6},
       URL = {https://doi.org/10.1007/s00222-021-01040-6},
}

@article {DHPW23,
    AUTHOR = {Dai, Xianzhe and Honda, Shouhei and Pan, Jiayin and Wei,
              Guofang},
     TITLE = {Singular {W}eyl's law with {R}icci curvature bounded below},
   JOURNAL = {Trans. Amer. Math. Soc. Ser. B},
  FJOURNAL = {Transactions of the American Mathematical Society. Series B},
    VOLUME = {10},
      YEAR = {2023},
     PAGES = {1212--1253},
      ISSN = {2330-0000},
   MRCLASS = {53C23 (53C17 53C21)},
  MRNUMBER = {4634191},
MRREVIEWER = {Yaoting\ Gui},
       DOI = {10.1090/btran/160},
       URL = {https://doi.org/10.1090/btran/160},
}

@article {Den25,
    AUTHOR = {Deng, Qin},
     TITLE = {H\"older continuity of tangent cones in {${\rm RCD}(K,N)$}
              spaces and applications to nonbranching},
   JOURNAL = {Geom. Topol.},
  FJOURNAL = {Geometry \& Topology},
    VOLUME = {29},
      YEAR = {2025},
    NUMBER = {2},
     PAGES = {1037--1114},
      ISSN = {1465-3060,1364-0380},
   MRCLASS = {53C23 (53C22)},
  MRNUMBER = {4900035},
MRREVIEWER = {Luis\ Guijarro},
       DOI = {10.2140/gt.2025.29.1037},
       URL = {https://doi.org/10.2140/gt.2025.29.1037},
}

@article {De16,
    AUTHOR = {Descombes, Dominic},
     TITLE = {Asymptotic rank of spaces with bicombings},
   JOURNAL = {Math. Z.},
  FJOURNAL = {Mathematische Zeitschrift},
    VOLUME = {284},
      YEAR = {2016},
    NUMBER = {3-4},
     PAGES = {947--960},
      ISSN = {0025-5874,1432-1823},
   MRCLASS = {53C23},
  MRNUMBER = {3563261},
MRREVIEWER = {Kyle\ Edward\ Kinneberg},
       DOI = {10.1007/s00209-016-1680-3},
       URL = {https://doi.org/10.1007/s00209-016-1680-3},
}

@article {DL15,
    AUTHOR = {Descombes, Dominic and Lang, Urs},
     TITLE = {Convex geodesic bicombings and hyperbolicity},
   JOURNAL = {Geom. Dedicata},
  FJOURNAL = {Geometriae Dedicata},
    VOLUME = {177},
      YEAR = {2015},
     PAGES = {367--384},
      ISSN = {0046-5755,1572-9168},
   MRCLASS = {53C23 (20F65 20F67)},
  MRNUMBER = {3370039},
MRREVIEWER = {Igor\ Belegradek},
       DOI = {10.1007/s10711-014-9994-y},
       URL = {https://doi.org/10.1007/s10711-014-9994-y},
}

@article {DL16,
    AUTHOR = {Descombes, Dominic and Lang, Urs},
     TITLE = {Flats in spaces with convex geodesic bicombings},
   JOURNAL = {Anal. Geom. Metr. Spaces},
  FJOURNAL = {Analysis and Geometry in Metric Spaces},
    VOLUME = {4},
      YEAR = {2016},
    NUMBER = {1},
     PAGES = {68--84},
      ISSN = {2299-3274},
   MRCLASS = {53C23 (20F65 20F67)},
  MRNUMBER = {3483604},
MRREVIEWER = {R\'emi\ Bernard\ Coulon},
       DOI = {10.1515/agms-2016-0003},
       URL = {https://doi.org/10.1515/agms-2016-0003},
}

@article {DiGigPasSou21,
    AUTHOR = {Di Marino, Simone and Gigli, Nicola and Pasqualetto, Enrico
              and Soultanis, Elefterios},
     TITLE = {Infinitesimal {H}ilbertianity of locally {${\rm
              CAT}(\kappa)$}-spaces},
   JOURNAL = {J. Geom. Anal.},
  FJOURNAL = {Journal of Geometric Analysis},
    VOLUME = {31},
      YEAR = {2021},
    NUMBER = {8},
     PAGES = {7621--7685},
      ISSN = {1050-6926,1559-002X},
   MRCLASS = {53C23 (46E35)},
  MRNUMBER = {4293907},
MRREVIEWER = {Padmavati},
       DOI = {10.1007/s12220-020-00543-7},
       URL = {https://doi.org/10.1007/s12220-020-00543-7},
}

@misc{FG25,
      title={{T}opological regularity of {B}usemann spaces of nonpositive curvature}, 
      author={Fujioka, T. and Gu, S.},
      year={2025},
      eprint={2504.14455v2},
      archivePrefix={arXiv},
      primaryClass={math.MG},
      note={Preprint}
}

@article {Fuk24,
    AUTHOR = {Fukaya, Tomohiro},
     TITLE = {A topological product decomposition of {B}usemann space},
   JOURNAL = {J. Math. Soc. Japan},
  FJOURNAL = {Journal of the Mathematical Society of Japan},
    VOLUME = {76},
      YEAR = {2024},
    NUMBER = {1},
     PAGES = {269--281},
      ISSN = {0025-5645,1881-1167},
   MRCLASS = {51F30 (53C23)},
  MRNUMBER = {4693873},
MRREVIEWER = {Logan\ S.\ Fox},
       DOI = {10.2969/jmsj/89738973},
       URL = {https://doi.org/10.2969/jmsj/89738973},
}

@article {Gig12,
    AUTHOR = {Gigli, Nicola},
     TITLE = {On the differential structure of metric measure spaces and
              applications},
   JOURNAL = {Mem. Amer. Math. Soc.},
  FJOURNAL = {Memoirs of the American Mathematical Society},
    VOLUME = {236},
      YEAR = {2015},
    NUMBER = {1113},
     PAGES = {vi+91},
      ISSN = {0065-9266,1947-6221},
      ISBN = {978-1-4704-1420-7},
   MRCLASS = {53C23 (30L05 30L10 49Q15 58C20)},
  MRNUMBER = {3381131},
MRREVIEWER = {Davide\ Vittone},
       DOI = {10.1090/memo/1113},
       URL = {https://doi.org/10.1090/memo/1113},
}

@misc{Gig13,
      title={The splitting theorem in non-smooth context}, 
      author={Gigli, N.},
      year={2013},
      eprint={1302.5555v1},
      archivePrefix={arXiv},
      primaryClass={math.MG},
      note={Preprint}
}

@inproceedings {Gr81,
    AUTHOR = {Gromov, M.},
     TITLE = {Hyperbolic manifolds, groups and actions},
 BOOKTITLE = {Riemann surfaces and related topics: {P}roceedings of the 1978
              {S}tony {B}rook {C}onference ({S}tate {U}niv. {N}ew {Y}ork,
              {S}tony {B}rook, {N}.{Y}., 1978)},
    SERIES = {Ann. of Math. Stud.},
    VOLUME = {No. 97},
     PAGES = {183--213},
 PUBLISHER = {Princeton Univ. Press, Princeton, NJ},
      YEAR = {1981},
      ISBN = {0-691-08264-2},
   MRCLASS = {53C15 (53C45 58F17)},
  MRNUMBER = {624814},
MRREVIEWER = {Elmer\ G.\ Rees},
}

@misc{HY25,
      title={On the structure of {B}usemann spaces with non-negative curvature}, 
      author={Han, B.-X. and Yin, L.},
      year={2025},
      eprint={2508.12348v2},
      archivePrefix={arXiv},
      primaryClass={math.MG},
      note={Preprint}
}

@article {HaL23,
    AUTHOR = {Hakavuori, Eero and Le Donne, Enrico},
     TITLE = {Blowups and blowdowns of geodesics in {C}arnot groups},
   JOURNAL = {J. Differential Geom.},
  FJOURNAL = {Journal of Differential Geometry},
    VOLUME = {123},
      YEAR = {2023},
    NUMBER = {2},
     PAGES = {267--310},
      ISSN = {0022-040X,1945-743X},
   MRCLASS = {53C17 (28A75 49K21 53C22)},
  MRNUMBER = {4571803},
MRREVIEWER = {Scott\ Robert\ Zimmerman},
       DOI = {10.4310/jdg/1680883578},
       URL = {https://doi.org/10.4310/jdg/1680883578},
}

@article {HiL07,
    AUTHOR = {Hitzelberger, Petra and Lytchak, Alexander},
     TITLE = {Spaces with many affine functions},
   JOURNAL = {Proc. Amer. Math. Soc.},
  FJOURNAL = {Proceedings of the American Mathematical Society},
    VOLUME = {135},
      YEAR = {2007},
    NUMBER = {7},
     PAGES = {2263--2271},
      ISSN = {0002-9939,1088-6826},
   MRCLASS = {53C23},
  MRNUMBER = {2299504},
MRREVIEWER = {Mario\ Bonk},
       DOI = {10.1090/S0002-9939-07-08728-X},
       URL = {https://doi.org/10.1090/S0002-9939-07-08728-X},
}

@article {Ho24,
    AUTHOR = {Honda, Shouhei},
     TITLE = {Metric measure spaces with {R}icci bounds from below},
   JOURNAL = {Sugaku Expositions},
  FJOURNAL = {Sugaku Expositions},
    VOLUME = {37},
      YEAR = {2024},
    NUMBER = {2},
     PAGES = {179--201},
      ISSN = {0898-9583,2473-585X},
   MRCLASS = {53C20 (53C21)},
  MRNUMBER = {4818014},
addendum= {Translation of \textit{S\=ugaku} \textbf{72}.2 (2020), 158--181}
}

@article {An14,
    AUTHOR = {Andreev, P. D.},
     TITLE = {Proof of the {B}usemann conjecture for {$G$}-spaces of
              nonpositive curvature},
   JOURNAL = {Algebra i Analiz},
  FJOURNAL = {Rossi\u iskaya Akademiya Nauk. Algebra i Analiz},
    VOLUME = {26},
      YEAR = {2014},
    NUMBER = {2},
     PAGES = {1--20},
      ISSN = {0234-0852},
   MRCLASS = {53C70},
  MRNUMBER = {3242034},
       DOI = {10.1090/S1061-0022-2015-01336-8},
       URL = {https://doi.org/10.1090/S1061-0022-2015-01336-8},
addendum = {Translation in \textit{St. Petersburg Math. J.} \textbf{26}.2 (2015), 193--206}
}

@article {IL19,
    AUTHOR = {Ivanov, Sergei and Lytchak, Alexander},
     TITLE = {Rigidity of {B}usemann convex {F}insler metrics},
   JOURNAL = {Comment. Math. Helv.},
  FJOURNAL = {Commentarii Mathematici Helvetici. A Journal of the Swiss
              Mathematical Society},
    VOLUME = {94},
      YEAR = {2019},
    NUMBER = {4},
     PAGES = {855--868},
      ISSN = {0010-2571,1420-8946},
   MRCLASS = {53B40 (53C23 53C60)},
  MRNUMBER = {4046007},
MRREVIEWER = {Wei\ Zhao},
       DOI = {10.4171/cmh/476},
       URL = {https://doi.org/10.4171/cmh/476},
}

@article {Jui09,
    AUTHOR = {Juillet, Nicolas},
     TITLE = {Geometric inequalities and generalized {R}icci bounds in the
              {H}eisenberg group},
   JOURNAL = {Int. Math. Res. Not. IMRN},
  FJOURNAL = {International Mathematics Research Notices. IMRN},
      YEAR = {2009},
    NUMBER = {13},
     PAGES = {2347--2373},
      ISSN = {1073-7928,1687-0247},
   MRCLASS = {53C23 (49Q05)},
  MRNUMBER = {2520783},
MRREVIEWER = {Alessio\ Figalli},
       DOI = {10.1093/imrn/rnp019},
       URL = {https://doi.org/10.1093/imrn/rnp019},
}

@article {KK19,
    AUTHOR = {Kapovitch, Vitali and Ketterer, Christian},
     TITLE = {Weakly noncollapsed {RCD} spaces with upper curvature bounds},
   JOURNAL = {Anal. Geom. Metr. Spaces},
  FJOURNAL = {Analysis and Geometry in Metric Spaces},
    VOLUME = {7},
      YEAR = {2019},
    NUMBER = {1},
     PAGES = {197--211},
      ISSN = {2299-3274},
   MRCLASS = {53C20 (53C21 53C23)},
  MRNUMBER = {4034631},
MRREVIEWER = {Shouhei\ Honda},
       DOI = {10.1515/agms-2019-0010},
       URL = {https://doi.org/10.1515/agms-2019-0010},
}

@article {KK20,
    AUTHOR = {Kapovitch, Vitali and Ketterer, Christian},
     TITLE = {C{D} meets {CAT}},
   JOURNAL = {J. Reine Angew. Math.},
  FJOURNAL = {Journal f\"ur die Reine und Angewandte Mathematik. [Crelle's
              Journal]},
    VOLUME = {766},
      YEAR = {2020},
     PAGES = {1--44},
      ISSN = {0075-4102,1435-5345},
   MRCLASS = {53C23},
  MRNUMBER = {4145200},
MRREVIEWER = {Loreno\ Heer},
       DOI = {10.1515/crelle-2019-0021},
       URL = {https://doi.org/10.1515/crelle-2019-0021},
}

@article {KKK22,
    AUTHOR = {Kapovitch, Vitali and Kell, Martin and Ketterer, Christian},
     TITLE = {On the structure of {RCD} spaces with upper curvature bounds},
   JOURNAL = {Math. Z.},
  FJOURNAL = {Mathematische Zeitschrift},
    VOLUME = {301},
      YEAR = {2022},
    NUMBER = {4},
     PAGES = {3469--3502},
      ISSN = {0025-5874,1432-1823},
   MRCLASS = {53C21 (53C20)},
  MRNUMBER = {4449717},
MRREVIEWER = {Wei\ Zhao},
       DOI = {10.1007/s00209-022-03015-6},
       URL = {https://doi.org/10.1007/s00209-022-03015-6},
}

@article {Ke19,
    AUTHOR = {Kell, Martin},
     TITLE = {Sectional curvature-type conditions on metric spaces},
   JOURNAL = {J. Geom. Anal.},
  FJOURNAL = {Journal of Geometric Analysis},
    VOLUME = {29},
      YEAR = {2019},
    NUMBER = {1},
     PAGES = {616--655},
      ISSN = {1050-6926,1559-002X},
   MRCLASS = {53C23 (51F99)},
  MRNUMBER = {3897028},
MRREVIEWER = {Barry\ Minemyer},
       DOI = {10.1007/s12220-018-0013-7},
       URL = {https://doi.org/10.1007/s12220-018-0013-7},
}

@article {KR15,
    AUTHOR = {Ketterer, Christian and Rajala, Tapio},
     TITLE = {Failure of topological rigidity results for the measure
              contraction property},
   JOURNAL = {Potential Anal.},
  FJOURNAL = {Potential Analysis. An International Journal Devoted to the
              Interactions between Potential Theory, Probability Theory,
              Geometry and Functional Analysis},
    VOLUME = {42},
      YEAR = {2015},
    NUMBER = {3},
     PAGES = {645--655},
      ISSN = {0926-2601,1572-929X},
   MRCLASS = {53C23 (28A33 49Q20)},
  MRNUMBER = {3336992},
MRREVIEWER = {Andrew\ Bucki},
       DOI = {10.1007/s11118-014-9450-5},
       URL = {https://doi.org/10.1007/s11118-014-9450-5},
}

@article {Kl99,
    AUTHOR = {Kleiner, Bruce},
     TITLE = {The local structure of length spaces with curvature bounded
              above},
   JOURNAL = {Math. Z.},
  FJOURNAL = {Mathematische Zeitschrift},
    VOLUME = {231},
      YEAR = {1999},
    NUMBER = {3},
     PAGES = {409--456},
      ISSN = {0025-5874,1432-1823},
   MRCLASS = {53C23 (57M50)},
  MRNUMBER = {1704987},
MRREVIEWER = {Raul\ Quiroga-Barranco},
       DOI = {10.1007/PL00004738},
       URL = {https://doi.org/10.1007/PL00004738},
}

@article {Kr11,
    AUTHOR = {Kramer, Linus},
     TITLE = {On the local structure and the homology of {${\rm
              CAT}(\kappa)$} spaces and {E}uclidean buildings},
   JOURNAL = {Adv. Geom.},
  FJOURNAL = {Advances in Geometry},
    VOLUME = {11},
      YEAR = {2011},
    NUMBER = {2},
     PAGES = {347--369},
      ISSN = {1615-715X,1615-7168},
   MRCLASS = {54C55 (51E24 57M50)},
  MRNUMBER = {2795430},
MRREVIEWER = {Pierre-Emmanuel\ Caprace},
       DOI = {10.1515/ADVGEOM.2010.049},
       URL = {https://doi.org/10.1515/ADVGEOM.2010.049},
}

@article {KVK04,
    AUTHOR = {Krist\'aly, Alexandru and Varga, Csaba and Kozma, L\'aszl\'o},
     TITLE = {The dispersing of geodesics in {B}erwald spaces of
              non-positive flag curvature},
   JOURNAL = {Houston J. Math.},
  FJOURNAL = {Houston Journal of Mathematics},
    VOLUME = {30},
      YEAR = {2004},
    NUMBER = {2},
     PAGES = {413--420},
      ISSN = {0362-1588},
   MRCLASS = {53C60 (53C22 53C70)},
  MRNUMBER = {2084910},
MRREVIEWER = {Ioan\ Radu\ Peter},
}

@article {KK06,
    AUTHOR = {Krist\'aly, Alexandru and Kozma, L\'aszl\'o},
     TITLE = {Metric characterization of {B}erwald spaces of non-positive
              flag curvature},
   JOURNAL = {J. Geom. Phys.},
  FJOURNAL = {Journal of Geometry and Physics},
    VOLUME = {56},
      YEAR = {2006},
    NUMBER = {8},
     PAGES = {1257--1270},
      ISSN = {0393-0440,1879-1662},
   MRCLASS = {53C70 (53C23 53C60)},
  MRNUMBER = {2234441},
MRREVIEWER = {Ioan\ Bucataru},
       DOI = {10.1016/j.geomphys.2005.06.014},
       URL = {https://doi.org/10.1016/j.geomphys.2005.06.014},
}

@article {Le11,
    AUTHOR = {Le Donne, Enrico},
     TITLE = {Metric spaces with unique tangents},
   JOURNAL = {Ann. Acad. Sci. Fenn. Math.},
  FJOURNAL = {Annales Academi\ae\ Scientiarum Fennic\ae. Mathematica},
    VOLUME = {36},
      YEAR = {2011},
    NUMBER = {2},
     PAGES = {683--694},
      ISSN = {1239-629X,1798-2383},
   MRCLASS = {54E35 (26A16 53C17)},
  MRNUMBER = {2865538},
MRREVIEWER = {Oleksiy\ A.\ Dovgoshey},
       DOI = {10.5186/aasfm.2011.3636},
       URL = {https://doi.org/10.5186/aasfm.2011.3636},
}

@article {Le15,
    AUTHOR = {Le Donne, Enrico},
     TITLE = {A metric characterization of {C}arnot groups},
   JOURNAL = {Proc. Amer. Math. Soc.},
  FJOURNAL = {Proceedings of the American Mathematical Society},
    VOLUME = {143},
      YEAR = {2015},
    NUMBER = {2},
     PAGES = {845--849},
      ISSN = {0002-9939,1088-6826},
   MRCLASS = {53C17 (22E25 53C60 58D19)},
  MRNUMBER = {3283670},
MRREVIEWER = {Gareth\ Speight},
       DOI = {10.1090/S0002-9939-2014-12244-1},
       URL = {https://doi.org/10.1090/S0002-9939-2014-12244-1},
}

@article {LotVil09,
    AUTHOR = {Lott, John and Villani, C\'edric},
     TITLE = {Ricci curvature for metric-measure spaces via optimal
              transport},
   JOURNAL = {Ann. of Math. (2)},
  FJOURNAL = {Annals of Mathematics. Second Series},
    VOLUME = {169},
      YEAR = {2009},
    NUMBER = {3},
     PAGES = {903--991},
      ISSN = {0003-486X,1939-8980},
   MRCLASS = {53C23 (49Q15)},
  MRNUMBER = {2480619},
MRREVIEWER = {Alessio\ Figalli},
       DOI = {10.4007/annals.2009.169.903},
       URL = {https://doi.org/10.4007/annals.2009.169.903},
}

@article {LN19,
    AUTHOR = {Lytchak, Alexander and Nagano, Koichi},
     TITLE = {Geodesically complete spaces with an upper curvature bound},
   JOURNAL = {Geom. Funct. Anal.},
  FJOURNAL = {Geometric and Functional Analysis},
    VOLUME = {29},
      YEAR = {2019},
    NUMBER = {1},
     PAGES = {295--342},
      ISSN = {1016-443X,1420-8970},
   MRCLASS = {53C23 (53C20 53C21)},
  MRNUMBER = {3925112},
       DOI = {10.1007/s00039-019-00483-7},
       URL = {https://doi.org/10.1007/s00039-019-00483-7},
}

@article {LN22,
    AUTHOR = {Lytchak, Alexander and Nagano, Koichi},
     TITLE = {Topological regularity of spaces with an upper curvature
              bound},
   JOURNAL = {J. Eur. Math. Soc. (JEMS)},
  FJOURNAL = {Journal of the European Mathematical Society (JEMS)},
    VOLUME = {24},
      YEAR = {2022},
    NUMBER = {1},
     PAGES = {137--165},
      ISSN = {1435-9855,1435-9863},
   MRCLASS = {53C23 (54E45)},
  MRNUMBER = {4375449},
MRREVIEWER = {Bo\.zena\ Pi\polhk atek},
       DOI = {10.4171/jems/1091},
       URL = {https://doi.org/10.4171/jems/1091},
}

@article {LS07,
    AUTHOR = {Lytchak, Alexander and Schroeder, Viktor},
     TITLE = {Affine functions on {${\rm CAT}(\kappa)$}-spaces},
   JOURNAL = {Math. Z.},
  FJOURNAL = {Mathematische Zeitschrift},
    VOLUME = {255},
      YEAR = {2007},
    NUMBER = {2},
     PAGES = {231--244},
      ISSN = {0025-5874,1432-1823},
   MRCLASS = {53C21 (53C20)},
  MRNUMBER = {2262730},
MRREVIEWER = {Mario\ Bonk},
       DOI = {10.1007/s00209-006-0020-4},
       URL = {https://doi.org/10.1007/s00209-006-0020-4},
}

@misc{Ma23,
      title={Examples of ${CD}(0,N)$ spaces with non-constant dimension}, 
      author={Magnabosco, Mattia},
      year={2023},
      eprint={2310.05738v1},
      archivePrefix={arXiv},
      primaryClass={math.MG},
      note={To appear in \textit{Ann. Sc. Norm. Super. Pisa Cl. Sci.}}
}

@misc{MMR25,
      title={On the rectifiability of $\mathsf{CD}(K,N)$ and $\mathsf{MCP}(K,N)$ with unique tangents}, 
      author={Magnabosco, Mattia and Mondino, Andrea and Rossi, Tommaso},
      year={2025},
      eprint={2505.01151v1},
      archivePrefix={arXiv},
      primaryClass={math.MG},
      note={Preprint}
}

@article {MN19,
    AUTHOR = {Mondino, Andrea and Naber, Aaron},
     TITLE = {Structure theory of metric measure spaces with lower {R}icci
              curvature bounds},
   JOURNAL = {J. Eur. Math. Soc. (JEMS)},
  FJOURNAL = {Journal of the European Mathematical Society (JEMS)},
    VOLUME = {21},
      YEAR = {2019},
    NUMBER = {6},
     PAGES = {1809--1854},
      ISSN = {1435-9855,1435-9863},
   MRCLASS = {53C23 (53C21)},
  MRNUMBER = {3945743},
MRREVIEWER = {Fernando\ Galaz-Garc\'ia},
       DOI = {10.4171/JEMS/874},
       URL = {https://doi.org/10.4171/JEMS/874},
}

@misc{NavPan25,
      title={Universal non-{CD} of sub-{R}iemannian manifolds}, 
      author={Navarro, Dimitri and Pan, Jiayin},
      year={2026},
      eprint={2507.00471v2},
      archivePrefix={arXiv},
      primaryClass={math.MG},
      note={To appear in \textit{J. Reine Angew. Math.}}
}

@article {Oh07,
    AUTHOR = {Ohta, Shin-ichi},
     TITLE = {On the measure contraction property of metric measure spaces},
   JOURNAL = {Comment. Math. Helv.},
  FJOURNAL = {Commentarii Mathematici Helvetici. A Journal of the Swiss
              Mathematical Society},
    VOLUME = {82},
      YEAR = {2007},
    NUMBER = {4},
     PAGES = {805--828},
      ISSN = {0010-2571,1420-8946},
   MRCLASS = {53C23 (28C15)},
  MRNUMBER = {2341840},
MRREVIEWER = {Jana\ Bj\"orn},
       DOI = {10.4171/CMH/110},
       URL = {https://doi.org/10.4171/CMH/110},
}

@article {Oh09,
    AUTHOR = {Ohta, Shin-ichi},
     TITLE = {Finsler interpolation inequalities},
   JOURNAL = {Calc. Var. Partial Differential Equations},
  FJOURNAL = {Calculus of Variations and Partial Differential Equations},
    VOLUME = {36},
      YEAR = {2009},
    NUMBER = {2},
     PAGES = {211--249},
      ISSN = {0944-2669,1432-0835},
   MRCLASS = {58E35 (53C60)},
  MRNUMBER = {2546027},
MRREVIEWER = {Giovanni\ Pisante},
       DOI = {10.1007/s00526-009-0227-4},
       URL = {https://doi.org/10.1007/s00526-009-0227-4},
}

@book {Pa14,
    AUTHOR = {Papadopoulos, Athanase},
     TITLE = {Metric spaces, convexity and non-positive curvature},
    SERIES = {IRMA Lectures in Mathematics and Theoretical Physics},
    VOLUME = {6},
   EDITION = {Second},
 PUBLISHER = {European Mathematical Society (EMS), Z\"urich},
      YEAR = {2014},
     PAGES = {xii+309},
      ISBN = {978-3-03719-132-3},
   MRCLASS = {53-01 (53C21 53C70 54E35 57M50)},
  MRNUMBER = {3156529},
       DOI = {10.4171/132},
       URL = {https://doi.org/10.4171/132},
}

@article{Po98,
 author = {Pogorelov, A. V.},
 title = {Busemann regular {{\(G\)}}-spaces},
 fjournal = {Reviews in Mathematics and Mathematical Physics},
 journal = {Rev. Math. Math. Phys.},
 issn = {1024-5278},
 volume = {10},
 number = {4},
 pages = {1--99},
 year = {1998},
 language = {English},
 zbMATH = {1386158},
 Zbl = {0996.53001}
}

@article {Riz18,
    AUTHOR = {Rizzi, Luca},
     TITLE = {A counterexample to gluing theorems for {MCP} metric measure
              spaces},
   JOURNAL = {Bull. Lond. Math. Soc.},
  FJOURNAL = {Bulletin of the London Mathematical Society},
    VOLUME = {50},
      YEAR = {2018},
    NUMBER = {5},
     PAGES = {781--790},
      ISSN = {0024-6093,1469-2120},
   MRCLASS = {53C17 (53C23 54E50)},
  MRNUMBER = {3873493},
MRREVIEWER = {Alessandro\ Ottazzi},
       DOI = {10.1112/blms.12186},
       URL = {https://doi.org/10.1112/blms.12186},
}

@article {Ri16,
    AUTHOR = {Rizzi, Luca},
     TITLE = {Measure contraction properties of {C}arnot groups},
   JOURNAL = {Calc. Var. Partial Differential Equations},
  FJOURNAL = {Calculus of Variations and Partial Differential Equations},
    VOLUME = {55},
      YEAR = {2016},
    NUMBER = {3},
     PAGES = {Art. 60, 20},
      ISSN = {0944-2669,1432-0835},
   MRCLASS = {53C17 (35R03 53C21 53C22 53C23 54E35)},
  MRNUMBER = {3502622},
MRREVIEWER = {Enrico\ Le Donne},
       DOI = {10.1007/s00526-016-1002-y},
       URL = {https://doi.org/10.1007/s00526-016-1002-y},
}

@article {St06,
    AUTHOR = {Sturm, Karl-Theodor},
     TITLE = {On the geometry of metric measure spaces. {II}},
   JOURNAL = {Acta Math.},
  FJOURNAL = {Acta Mathematica},
    VOLUME = {196},
      YEAR = {2006},
    NUMBER = {1},
     PAGES = {133--177},
      ISSN = {0001-5962,1871-2509},
   MRCLASS = {53C23},
  MRNUMBER = {2237207},
MRREVIEWER = {Juha\ Heinonen},
       DOI = {10.1007/s11511-006-0003-7},
       URL = {https://doi.org/10.1007/s11511-006-0003-7},
}

@book {Vil,
    AUTHOR = {Villani, C\'edric},
     TITLE = {Optimal transport Old and new},
    SERIES = {Grundlehren der mathematischen Wissenschaften [Fundamental
              Principles of Mathematical Sciences]},
    VOLUME = {338},
 PUBLISHER = {Springer-Verlag, Berlin},
      YEAR = {2009},
     PAGES = {xxii+973},
      ISBN = {978-3-540-71049-3},
   MRCLASS = {49-02 (28A75 37J50 49Q20 53C23 58E30)},
  MRNUMBER = {2459454},
MRREVIEWER = {Dario\ Cordero-Erausquin},
       DOI = {10.1007/978-3-540-71050-9},
       URL = {https://doi.org/10.1007/978-3-540-71050-9},
}

@article {vo08,
    AUTHOR = {{von Renesse}, Max-K.},
     TITLE = {On local {P}oincar\'e{} via transportation},
   JOURNAL = {Math. Z.},
  FJOURNAL = {Mathematische Zeitschrift},
    VOLUME = {259},
      YEAR = {2008},
    NUMBER = {1},
     PAGES = {21--31},
      ISSN = {0025-5874,1432-1823},
   MRCLASS = {53C21 (46E35 53C23 54E35)},
  MRNUMBER = {2375612},
MRREVIEWER = {Filippo\ Santambrogio},
       DOI = {10.1007/s00209-007-0206-4},
       URL = {https://doi.org/10.1007/s00209-007-0206-4},
}

@article {Po90,
    AUTHOR = {Pogorelov, A. V.},
     TITLE = {Regular {$G$}-spaces of {H}. {B}usemann},
   JOURNAL = {Dokl. Akad. Nauk SSSR},
  FJOURNAL = {Doklady Akademii Nauk SSSR},
    VOLUME = {314},
      YEAR = {1990},
    NUMBER = {1},
     PAGES = {114--118},
      ISSN = {0002-3264},
   MRCLASS = {53C70 (53C60)},
  MRNUMBER = {1118490},
addendum = {Translation in \textit{Soviet Math. Dokl.} \textbf{42}.2 (1991), 356--359}
}

@incollection {Jo,
    AUTHOR = {John, Fritz},
     TITLE = {Extremum problems with inequalities as subsidiary conditions},
 BOOKTITLE = {Studies and {E}ssays {P}resented to {R}. {C}ourant on his 60th
              {B}irthday, {J}anuary 8, 1948},
     PAGES = {187--204},
 PUBLISHER = {Interscience Publishers, New York},
      YEAR = {1948},
   MRCLASS = {49.0X},
  MRNUMBER = {30135},
MRREVIEWER = {J.\ E.\ Wilkins, Jr.},
}

@article {BGP92,
    AUTHOR = {Burago, Yu. and Gromov, M. and Perel'man, G.},
     TITLE = {A. {D}. {A}leksandrov spaces with curvatures bounded below},
   JOURNAL = {Uspekhi Mat. Nauk},
  FJOURNAL = {Uspekhi Matematicheskikh Nauk},
    VOLUME = {47},
      YEAR = {1992},
    NUMBER = {2(284)},
     PAGES = {3--51, 222},
      ISSN = {0042-1316,2305-2872},
   MRCLASS = {53C21 (53C23)},
  MRNUMBER = {1185284},
MRREVIEWER = {Tadeusz\ Januszkiewicz},
       DOI = {10.1070/RM1992v047n02ABEH000877},
       URL = {https://doi.org/10.1070/RM1992v047n02ABEH000877},
addendum = {Translation in \textit{Russian Math. Surveys} \textbf{47}.2 (1992), 1--58}
}

@article{Rif13,
title = {Ricci curvatures in Carnot groups},
journal = {Mathematical Control and Related Fields},
volume = {3},
number = {4},
pages = {467-487},
year = {2013},
issn = {2156-8472},
doi = {10.3934/mcrf.2013.3.467},
url = {https://www.aimsciences.org/article/id/06ad68b2-25fd-47a2-b684-d67eb5bfbd5f},
author = {Ludovic Rifford}
}

@misc{BorTas23,
      title={{Measure contraction property, curvature exponent and geodesic dimension of sub-Finsler $\ell^p$-Heisenberg groups}}, 
      author={Samuël Borza and Kenshiro Tashiro},
      year={2025},
      eprint={2305.16722v2},
      archivePrefix={arXiv},
      primaryClass={math.MG},
      note = {To appear in \textit{Ann. Inst. Fourier (Grenoble)}}
}

@book{Simbook,
    AUTHOR = {Simon, Leon},
     TITLE = {{Lectures on Geometric Measure Theory}},
    SERIES = {Proceedings of the Centre for Mathematical Analysis,
              Australian National University},
    VOLUME = {3},
 PUBLISHER = {Australian National University, Centre for Mathematical
              Analysis, Canberra},
      YEAR = {1983},
     PAGES = {vii+272},
      ISBN = {0-86784-429-9},
   MRCLASS = {49-01 (28A75 49F20)},
  MRNUMBER = {756417},
MRREVIEWER = {J.\ S.\ Joel},
}

@incollection {Be18,
    AUTHOR = {Berestovski\u i, Valeri\u i\ Nikolaevich},
     TITLE = {Busemann's results, ideas, questions and locally compact homogeneous geodesic spaces},
 BOOKTITLE = {Herbert Busemann Selected works. {I}},
     PAGES = {41--84},
 PUBLISHER = {Springer, Cham},
      YEAR = {2018},
      ISBN = {978-3-319-64294-9},
}

@article {On05,
    AUTHOR = {Ontaneda, Pedro},
     TITLE = {Cocompact {CAT}(0) spaces are almost geodesically complete},
   JOURNAL = {Topology},
  FJOURNAL = {Topology. An International Journal of Mathematics},
    VOLUME = {44},
      YEAR = {2005},
    NUMBER = {1},
     PAGES = {47--62},
      ISSN = {0040-9383},
   MRCLASS = {57M07 (53C22)},
  MRNUMBER = {2104000},
MRREVIEWER = {Craig\ R.\ Guilbault},
       DOI = {10.1016/j.top.2004.01.010},
       URL = {https://doi.org/10.1016/j.top.2004.01.010},
}

@article {GO07,
    AUTHOR = {Geoghegan, Ross and Ontaneda, Pedro},
     TITLE = {Boundaries of cocompact proper {${\rm CAT}(0)$} spaces},
   JOURNAL = {Topology},
  FJOURNAL = {Topology. An International Journal of Mathematics},
    VOLUME = {46},
      YEAR = {2007},
    NUMBER = {2},
     PAGES = {129--137},
      ISSN = {0040-9383},
   MRCLASS = {57M07 (20F65)},
  MRNUMBER = {2313068},
MRREVIEWER = {Robert\ W.\ Bell},
       DOI = {10.1016/j.top.2006.12.002},
       URL = {https://doi.org/10.1016/j.top.2006.12.002},
}

@article {KR21,
    AUTHOR = {Kent, Curtis and Ricks, Russell},
     TITLE = {Asymptotic cones and boundaries of {$\rm CAT(0)$} spaces},
   JOURNAL = {Indiana Univ. Math. J.},
  FJOURNAL = {Indiana University Mathematics Journal},
    VOLUME = {70},
      YEAR = {2021},
    NUMBER = {4},
     PAGES = {1441--1469},
      ISSN = {0022-2518,1943-5258},
   MRCLASS = {20F65 (53C23)},
  MRNUMBER = {4318480},
MRREVIEWER = {Igor\ Belegradek},
       DOI = {10.1512/iumj.2021.70.8552},
       URL = {https://doi.org/10.1512/iumj.2021.70.8552},
}

@incollection {BN93,
    AUTHOR = {Berestovskij, V. N. and Nikolaev, I. G.},
     TITLE = {Multidimensional generalized {R}iemannian spaces},
 BOOKTITLE = {Geometry, {IV}},
    SERIES = {Encyclopaedia Math. Sci.},
    VOLUME = {70},
     PAGES = {165--243, 245--250},
 PUBLISHER = {Springer, Berlin},
      YEAR = {1993},
      ISBN = {3-540-54701-0},
   MRCLASS = {53C20},
  MRNUMBER = {1263965},
       DOI = {10.1007/978-3-662-02897-1\_2},
       URL = {https://doi.org/10.1007/978-3-662-02897-1_2},
}

@article {Be02,
    AUTHOR = {Berestovski\u i, V. N.},
     TITLE = {Busemann spaces with upper-bounded {A}leksandrov curvature},
   JOURNAL = {Algebra i Analiz},
  FJOURNAL = {Rossi\u iskaya Akademiya Nauk. Algebra i Analiz},
    VOLUME = {14},
      YEAR = {2002},
    NUMBER = {5},
     PAGES = {3--18},
      ISSN = {0234-0852},
   MRCLASS = {53C70 (53C23)},
  MRNUMBER = {1970330},
  addendum = {Translation in \textit{St. Petersburg Math. J.} \textbf{14}.5 (2003), 713--723}
}

\end{document}